\documentclass[12pt]{amsart}

\usepackage{ucs}

\usepackage{amssymb}
\usepackage{amsthm}
\usepackage{amsmath}
\usepackage{latexsym}
\usepackage[cp1251]{inputenc}
\usepackage{graphicx}
\usepackage{wrapfig}
\usepackage{caption}
\usepackage{subcaption}
\usepackage{indentfirst}
\usepackage[left=2.4cm,right=2.4cm,top=2.4cm,bottom=2.4cm,bindingoffset=0cm]{geometry}
\usepackage{enumerate}
\usepackage{makecell}
\usepackage{xr}
\usepackage{cleveref}
\DeclareMathOperator{\aut}{Aut}

\DeclareMathOperator{\cyc}{Cyc}

\DeclareMathOperator{\id}{id}

\DeclareMathOperator{\orb}{Orb}

\DeclareMathOperator{\sym}{Sym}
\DeclareMathOperator{\rad}{rad}

\def\tm#1{\item[{\rm (#1)}]}

\makeatletter 
\def\@seccntformat#1{\csname the#1\endcsname. } 
\def\@biblabel#1{#1.} 

\newcommand{\overbar}[1]{\mkern 1.5mu\overline{\mkern-1.5mu#1\mkern-1.5mu}\mkern 1.5mu}

\title{Classification of abelian Schur groups II}

\author{Grigory Ryabov}

\address{School of Mathematical Sciences, Hebei Key Laboratory of Computational Mathematics and Applications, Hebei Normal University, Shijiazhuang 050024, P. R. China}

\address{Sobolev Institute of Mathematics, Novosibirsk, Russia}

\address{Novosibirsk State Technical University, Novosibirsk, Russia}

\email{gric2ryabov@gmail.com}

\thanks{The author was supported by the state contract of the Sobolev Institute of Mathematics (project number FWNF-2026-0011)}

\date{}

\newtheorem{prop}{Proposition}[section]

\newtheorem*{theo2}{Theorem (Classification of abelian Schur groups)}

\newtheorem{lemm}[prop]{Lemma}
\newtheorem{theo}[prop]{Theorem}

\newtheorem*{prob}{Problem (P\"{o}schel, 1974)}

\newtheorem*{corl1}{Corollary}
\theoremstyle{definition}

\newtheorem{rem}[prop]{Remark}

\newcommand{\aref}[1]{A\ref{#1}}

\begin{document}

\begin{abstract}
A finite group $G$ is called a \emph{Schur} group if every Schur ring over $G$ is \emph{schurian}, i.e. associated in a natural way with a subgroup of the symmetric group $\sym(G)$ that contains all right translations of~$G$. The list of all possible abelian Schur groups was obtained by Evdokimov, Kov\'acs, and Ponomarenko in 2016.  In two papers, we complete a classification of abelian Schur groups. In the present paper, we prove that several groups of nonpowerful order from the list are Schur groups. By that, we obtain a classification of abelian Schur groups.

\end{abstract}

\maketitle
\section{Introduction}

In the present paper, we continue to study the schurity problem for $S$-rings over abelian groups. Recall that an \emph{$S$-ring} (\emph{Schur ring}) over a finite group $G$ is a subring of the group ring $\mathbb{Z}G$ which is a free $\mathbb{Z}$-module spanned by a partition of $G$ closed under taking inverse and containing the identity element $e$ of $G$ as a class. The theory of $S$-rings goes back to the classical paper of Schur~\cite{Schur}. Later, the $S$-ring theory was substantially developed by Wielandt (see Chapter~IV of the monograph~\cite{Wi}). $S$-rings have many applications in algebraic combinatorics, in particular, in studying Cayley graphs, permutation group theory, representation theory etc. The interested reader can find more information and details on $S$-rings in the monograph~\cite{CP} and survey~\cite{MP0}

One of the sources of $S$-rings over~$G$ is permutation groups having a regular subgroup isomorphic to~$G$. Namely, the partition of~$G$ into the orbits of a one-point-stabilizer in such group defines an $S$-ring. As Wielandt wrote~\cite[p.~54]{Wi2}, Schur had conjectured that every $S$-ring can be constructed in such way. This was disproved by Wielandt~\cite{Wi}. An $S$-ring is said to be \emph{schurian} if it arises from the action of an appropriate permutation group having a regular subgroup and a finite group is defined to be a \emph{Schur} group if every $S$-ring over this group is schurian. The last two notions were introduced in~\cite{Po}. In the same paper, the following problem was posed.

\begin{prob}
Determine all Schur groups.
\end{prob}

The above problem seems to be hard, in particular, because the number of $S$-rings over a given group can be exponential in the order of the group. On the other hand, it is unclear in the general case how to find an appropriate permutation group for a given $S$-ring or to prove that there is no such group.

More information on schurian $S$-rings, Schur groups, and the P\"{o}schel problem is provided in the first paper on the classification of abelian Schur groups~\cite{Ry5} (see also~\cite{EKP1,EKP2}). Several results on schurity of nonabelian groups can be found in~\cite{MP,PV,Ry1,Ry3,Ry4}. It should be noted that almost all of them state that some nonabelian groups are non-Schur groups. On the other hand, there are several results~\cite{EKP1,EKP2,GNP,KP,Po,MP,PR,Ry2,Ry5} on schurity of abelian groups. Summarizing all of them and the computational results from~\cite{Ziv}, one can deduce that schurity of only the following abelian groups is unknown:
$$C_4\times C_{2p},~E_8\times C_p,~C_{6}\times C_{3^k},~E_9 \times C_{2q},$$
where $C_n$ and $E_n$ denote the cyclic and elementary abelian groups of order~$n$, respectively, $p\geq 11$ is an odd prime, $q\geq 5$ is a prime, and $k\geq 3$. In the present paper, we study schurity of the above groups. We divide them into two parts, namely, the groups of order $8p$, where $p$ is an odd prime, and the groups of twice odd order. The main results of the paper are two theorems below.

\begin{theo}\label{main1}
The groups $C_4\times C_{2p}$ and $E_8\times C_p$, where $p$ is an odd prime, are Schur groups. 
\end{theo}

\begin{theo}\label{main2}
The groups $C_{6}\times C_{3^k}$ and $E_9 \times C_{2q}$, where $k\geq 1$ and $q$ is a prime, are Schur groups.
\end{theo}

To prove Theorems~\ref{main1} and~\ref{main2}, we characterize all $S$-rings over the groups from these theorems. Namely, we prove the theorem below which immediately follows from Theorems~\ref{8p} and~\ref{2odd}.

\begin{theo}\label{main15}
Every nontrivial $S$-ring over one of the groups $C_4\times C_{2p}$, $E_8\times C_p$, $C_{6}\times C_{3^k}$, $E_9 \times C_{2q}$, where $p$ is an odd prime, $q$ is a prime, and $k\geq 1$, is cyclotomic or a nontrivial tensor or generalized wreath product. 
\end{theo}

A similar statement was proved for an arbitrary cyclic group in~\cite{LM1,LM2} and for some other abelian groups in~\cite{EKP2,MP0,PR,Ry2,Ry5}. On the other hand, the above statement does not hold for every abelian group (see~\cite{Wi}). It seems interesting to characterize all groups whose all nontrivial $S$-rings are cyclotomic or nontrivial tensor or generalized wreath products.

Summarizing the results of~\cite[Theorem~1.1]{EKP1},~\cite[Theorems~1.2-1.3]{EKP2},~\cite[Theorem~1.1]{MP},~\cite[Theorem~1.1]{PR},~\cite[Theorem~1.1]{Ry2},~\cite[Theorems~1.2-1.4]{Ry5}, and Theorems~\ref{main1} and~\ref{main2}, we obtain a classification of abelian Schur groups.

\begin{theo2}
An abelian group is a Schur group if and only if it is isomorphic to one of the following groups:

\begin{enumerate}

\tm{1} $C_n$, where $n\in\{p^k, pq^k, 2pq^k, pqr, 2pqr\}$, $p$, $q$, and $r$ are primes, and $k\geq 0$ is an integer;

\tm{2} $E_n$, where $n\in\{4,8,9,16,27,32\}$;

\tm{3} $C_2 \times C_{2^k}$, $C_{4}\times C_{2p}$, $E_4 \times C_{p^k}$, $E_4 \times C_{pq}$, $E_{16}\times C_3$, where $p\neq 2$ and $q$ are primes and $k\geq 1$ is an integer;

\tm{4} $C_3 \times C_{3^k}$, $C_{6}\times C_{3^k}$, $E_9\times C_{q}$, $E_9 \times C_{2q}$, where
$q$ is a prime and $k\geq 1$ is an integer.

\end{enumerate}

\end{theo2}

One can also deduce the statement below from the above classification, Theorem~\ref{main15}, the results of the papers~\cite{EKP2,LM1,LM2,MP0,PR,Ry2,Ry5}, and computations using the package~\cite{GAP}.

\begin{corl1}
Every nontrivial $S$-ring over an abelian Schur group is cyclotomic or a nontrivial tensor or generalized wreath product. 
\end{corl1}

We finish the introduction with a brief outline of the paper. Throughout the paper, we use the notation and preliminaries from the first paper on the classification of abelian Schur groups~\cite{Ry5}. In Section~$2$, we give some notation and preliminary statements which do not appear in~\cite{Ry5}. In Section~$3$, we study a structure of a basic set of an $S$-ring over an abelian group having a Sylow subgroup of prime order. Section~$4$ contains a description of $S$-rings over abelian groups of order~$8$. In Sections~$5$,  we provide a characterization of $S$-rings over the groups $C_4\times C_{2p}$ and $E_8\times C_p$. Theorem~\ref{main1} is proved in Sections~$6$. Section~$7$ contains necessary information on $S$-rings over the groups $C_3\times C_{3^k}$. A characterization of $S$-rings over the groups $C_{6}\times C_{3^k}$ and $E_9 \times C_{2q}$ is given in Section~$8$. Theorem~\ref{main2} is proved in Section~$9$.

\section{Preliminaries}

Throughout the text, we freely use basic definitions and facts from the $S$-ring theory. For a necessary background, we refer the reader to the first paper on the classification of abelian Schur groups~\cite{Ry5}, where most of the preliminary material is contained. In this paper, we follow the notation and terminology from~\cite{Ry5}. When referring to that paper, we keep only the number of the statement, preceding it by the letter A (e.g.,
instead of~\cite[Lemma~2.1]{Ry5}, we write Lemma~\aref{A-intersection}). Some additional or specific notation is listed below.

\vspace{2mm}

\noindent The set of all nontrivial elements of a group $G$ is denoted by $G^\#$.

\vspace{2mm}

\noindent The projections of $X\subseteq G_1\times G_2$ to $G_1$ and $G_2$ are denoted by $X_{G_1}$ and $X_{G_2}$, respectively.

\vspace{2mm}

\noindent The trivial $S$-ring over a group $G$ is denoted by $\mathcal{T}_G$.

\vspace{2mm}

\noindent If $\mathcal{A}$ is an $S$-ring over a group $G$ and $T$ is an $\mathcal{A}$-set, then the set of all basic sets of $\mathcal{A}$ inside $T$ is denoted by $\mathcal{S}(\mathcal{A})_T$.

\vspace{2mm}

\noindent  Given an abelian group $G$ and an $S$-ring $\mathcal{A}$ over $G$, the group and $S$-ring dual to $G$ and $\mathcal{A}$, respectively, are denoted by $\widehat{G}$ and $\widehat{\mathcal{A}}$, respectively. If $H\leq G$, then the image of $H$ under the uniquely determined lattice antiisomorphism between the subgroups of~$G$ and~$\widehat{G}$ is denoted by~$H^\bot$.

\vspace{2mm}

\noindent  Given $G=G_1\times G_2$, $K_i^0\trianglelefteq K_i\leq \aut(G_i)$, $i\in \{1,2\}$, such that $K_1/K_1^0\cong K_2/K_2^0$, and an isomorphism $\psi$ from $K_1/K_1^0$ to $K_2/K_2^0$, put 
$$K(K_1,K_1^0,K_2,K_2^0,\psi)=\{(\sigma,\tau)\in K_1\times K_2:~(\sigma^{\pi_1})^\psi=\tau^{\pi_2}\},$$ 
where $\pi_1$ and $\pi_2$ are the canonical epimorphisms from $K_1$ to $K_1/K_1^0$ and from $K_2$ to $K_2/K_2^0$, respectively.
\vspace{2mm}

We also provide in this section the statements which are used in the proofs of the results of the current paper and do not appear in~\cite{Ry5}.

\begin{lemm}\label{cayleymin}
Let $\mathcal{A}$ be an $S$-ring over an abelian group $G$ and $S=U/L$ an $\mathcal{A}$-section. Suppose that $\mathcal{A}$ is the $S$-wreath product and the $S$-rings $\mathcal{A}_U$ and $\mathcal{A}_{G/L}$ are schurian. Then $\mathcal{A}$ is schurian whenever $\mathcal{A}_U$ and $\mathcal{A}_{G/L}$ are cyclotomic and $\mathcal{A}_S$ is Cayley minimal.
\end{lemm}

\begin{proof}
Let $M_1\leq \aut(U)$ and $M_0\leq \aut(G/L)$ be such that $\mathcal{A}_U=\cyc(M_1,U)$ and $\mathcal{A}_{G/L}=\cyc(M_0,G/L)$. Then $\mathcal{A}_S=\cyc(M_1^S,S)=\cyc(M_0^S,S)$. Since $\mathcal{A}_S$ is Cayley minimal, we conclude that $M_1^S=M_0^S=\aut(\mathcal{A}_S)\cap \aut(S)$. Therefore $K_1^S=K_0^S$, where $K_1=U_r\rtimes M_1$ and $K_0=(G/L)_r\rtimes M_0$. Thus, $K_1$ and $K_0$ satisfy the conditions from Lemma~\aref{A-schurwr} and hence $\mathcal{A}$ is schurian by Lemma~\aref{A-schurwr}.
\end{proof}

The lemma below is a special case of~\cite[Theorem~22]{BC} which provides a sufficient condition of schurity for a crested product of association schemes.

\begin{lemm}\label{schurstar}
Let $\mathcal{A}$ be an $S$-ring over a group $G$. Suppose that $\mathcal{A}=\mathcal{A}_{L} \star \mathcal{A}_{U}$ for some $\mathcal{A}$-subgroups $L$ and $U$ and the $S$-rings $\mathcal{A}_L$ and $\mathcal{A}_{U/(L\cap U)}$ are schurian. Then $\mathcal{A}$ is also schurian.
\end{lemm}

\begin{lemm}\label{cyclpwreath}
Let $\mathcal{A}$ be a nontrivial generalized wreath product over a cyclic $p$-group~$G$. Then the following statements hold:
\begin{enumerate}

\tm{1} $\mathcal{A}$ is the nontrivial $U_1/L_1$-wreath product for an $\mathcal{A}$-section $U_1/L_1$ of $G$ such that $L_1$ is the least nontrivial $\mathcal{A}$-subgroup of $G$ and $|\rad(\mathcal{A}_{U_1})|=1$;

\tm{2} $\mathcal{A}$ is the nontrivial $U_0/L_0$-wreath product for an $\mathcal{A}$-section $U_0/L_0$ of $G$ such that $U_0$ is the greatest proper $\mathcal{A}$-subgroup of $G$ and $|\rad(\mathcal{A}_{G/L_0})|=1$.
\end{enumerate}
\end{lemm}

\begin{proof}
Statement~$(1)$ is Lemma~\aref{A-cyclpwreath}. Let us prove Statement~$(2)$. Since $G$ is a cyclic $p$-group, $G$ has the greatest proper $\mathcal{A}$-subgroup $U_0$. Let $L_0=\rad(\mathcal{A})$. Clearly,  $|\rad(\mathcal{A}_{G/L_0})|=1$. As $\mathcal{A}$ is a nontrivial generalized wreath product, we have $|L_0|>1$. By the definition of $U_0$, every basic set $X$ of $\mathcal{A}$ outside $U_0$ contains a generator of $G$ and hence $L_0\leq \rad(X)$. Therefore $\mathcal{A}$ is the nontrivial $U_0/L_0$-wreath product as desired.
\end{proof}

\section{Basic sets of a dense $S$-ring over $H\times C_p$}

Throughout this section, $G=H\times P$, where $H$ is an abelian group and $P\cong C_{p}$ with prime $p$ coprime to $|H|$, $\mathcal{A}$ is a \emph{dense} $S$-ring over $G$, i.e. $H$ and $P$ are $\mathcal{A}$-subgroups, and $X\in \mathcal{S}(\mathcal{A})_{G\setminus (H\cup P)}$. Lemma~\aref{A-tenspr}(1) implies that $X_H\in \mathcal{S}(\mathcal{A}_H)$ and $X_P\in \mathcal{S}(\mathcal{A}_P)$. The lemma below is a special case of Lemma~\aref{A-orbit}.

\begin{lemm}\label{orbit}
There exists $K\leq \aut(P)$ such that $X_h=h^{-1}X\cap P\in \orb(K,P)$ for every $h\in X_H$.
\end{lemm}

\begin{lemm}\label{conj}
Every $Y\in \mathcal{S}(\mathcal{A})_{X_HP^\#}$ is rationally conjugate to $X$.
\end{lemm}

\begin{proof}
Let $h\in X_H=Y_H$, $g_1\in X_P\subseteq P^\#$, and $g_2\in Y_P\subseteq P^\#$. Since $P\cong C_{p}$ and $|H|$ is coprime to~$p$, there exists a positive integer~$m$ such that $h^m=h$ and $g_1^m=g_2$. Therefore $hg_2=(hg_1)^m\in Y\cap X^{(m)}$. Lemma~\aref{A-burn} implies that $X^{(m)}\in \mathcal{S}(\mathcal{A})$ and hence $Y=X^{(m)}$ as required.
\end{proof}

Clearly, if $X\in \orb(K,G)$ for some $K\leq \aut(G)$ and $Y$ is rationally conjugate to $X$, then $Y\in \orb(K,G)$. Recall that a partition of a set is said to be \emph{uniform} if all its classes have the same cardinality.

\begin{lemm}\label{basesethp}
There exist uniform partitions $\Pi_H$ and $\Pi_P$ of $X_H$ and $X_P$, respectively, and a bijection $\varphi$ from $\Pi_H$ to $\Pi_P$ such that
\begin{equation}\label{xrepresent}
X=\bigcup \limits_{\Delta\in \Pi_H} \Delta\times\Delta^\varphi.
\end{equation}
\end{lemm}

\begin{proof}
Clearly, $X$ can be presented in the form 
$$X=\bigcup \limits_{g\in X_P} X_{H,g} g,$$
where $X_{H,g}\subseteq X_H$ and $\bigcup \limits_{g\in X_P} X_{H,g}=X_H$. Suppose that $X_{H,g}\cap X_{H,g^\prime}\neq \varnothing$ for some $g,g^\prime \in X_P$. Then $hg,hg^\prime\in X$ for some $h\in X_H$. Since $P\cong C_{p}$ and $|H|$ is coprime to~$p$, there exists a positive integer~$m$ such that $g^m=g^\prime$ and $h^m=h$. Therefore $(hg)^m=hg^\prime\in X^{(m)}\cap X$. Together with Lemma~\aref{A-burn}, this implies that $X^{(m)}=X$ and hence $X_{H,g}\subseteq X_{H,g^\prime}$. As $|X_{H,g}|=|X_{H,g^\prime}|$ by Lemma~\aref{A-intersection}, we obtain $X_{H,g}=X_{H,g^\prime}$.

The above discussion yields that $X_{H,g}=X_{H,g^\prime}$ or $X_{H,g}\cap X_{H,g^\prime}=\varnothing$ for all $g,g^\prime\in X_P$. So the binary relation consisting of all pairs $(g,g^\prime)\in X_P\times X_P$ such that $X_{H,g}=X_{H,g^\prime}$ is an equivalence relation on $X_P$. Let $\Pi_P$ be the partition of $X_P$ into the classes of this equivalence relation. Given $\Lambda\in \Pi_P$, the set $X_{H,g}$ does not depend on $g\in \Lambda$ by the definition of $\Pi_P$. For every $\Lambda\in \Pi_P$, put $\Delta(\Lambda)=X_{H,g}$, where $g\in \Lambda$. Clearly, $\Pi_H=\{\Delta(\Lambda):~\Lambda\in \Pi_P\}$ is a partition of $X_H$ and $\varphi:\Delta(\Lambda)\mapsto \Lambda$ is a bijection from $\Pi_H$ to $\Pi_P$ such that Eq.~\eqref{xrepresent} holds. The partitions $\Pi_H$ and $\Pi_P$ are uniform by Lemma~\aref{A-intersection} and we are done. 
\end{proof}

\begin{lemm}\label{orbitspip}
With the notation of Lemma~\ref{basesethp}, there exist groups $K_P^0\leq K_P\leq \aut(P)\cong C_{p-1}$ such that $\mathcal{A}_P=\cyc(K_P,P)$, $X_P\in \orb(K_P,P)$, $\Pi_P=\orb(K_P^0,X_P)$, and $|K_P/K_P^0|=|\Pi_P|$. 
\end{lemm}

\begin{proof}
Lemma~\aref{A-tenspr}(1) implies that $X_P\in \mathcal{S}(\mathcal{A}_P)$. Due to Lemma~\aref{A-cyclprime}, there exists $K_P\leq \aut(P)$ such that $\mathcal{A}_P=\cyc(K_P,P)$ and hence $X_P\in\orb(K_P,P)$. By Lemma~\ref{orbit}, there exists $K_P^0\leq \aut(P)$ such that $\Pi_P=\orb(K_P^0,X_P)$. Since $\aut(P)\cong C_{p-1}$ and $K_P$ and $K_P^0$ are semiregular on $P^\#$, we have $K_P^0\leq K_P$ and $|K_P/K_P^0|=|\Pi_P|$.
\end{proof}

One can see that $|K_P/K_P^0|=|\Pi_P|=|\Pi_H|$, where the first equality holds by Lemma~\ref{orbitspip}, whereas the second one holds by Lemma~\ref{basesethp}. If $K_0\trianglelefteq K\leq \sym(\Omega)$ and $\Delta\in \orb(K,\Omega)$, then every $\sigma\in K/K_0$ induces the permutation on the set $\Pi=\orb(K_0,\Delta)$. We denote this bijection by $\sigma^\Pi$ and avoid $\Pi$ when it is clear from the context.

\begin{lemm}\label{subdirect}
With the notation of Lemma~\ref{basesethp} and Lemma~\ref{orbitspip}, suppose that there exist groups $K_H^0\trianglelefteq K_H\leq \aut(H)$ such that $K_H/K_H^0\cong C_k$, where $k=|\Pi_H|$, $X_H\in \orb(K_H,H)$, and $\Pi_H=\orb(K_H^0,X_H)$, and an isomorphism $\psi$ from $K_H/K_H^0\cong C_k$ to $K_P/K_P^0\cong C_k$ such that
\begin{equation}\label{commutebij}
\varphi\circ(\sigma^\psi)^{\Pi_P}=\sigma^{\Pi_H}\circ\varphi
\end{equation}
for every $\sigma\in K_H/K_H^0$. Then $X\in \orb(K,G)$, where $K=K(K_H,K_H^0,K_P,K_P^0,\psi)$.
\end{lemm}

\begin{proof}
Note that the groups $(K_H/K_H^0)^{\Pi_H}$ and $(K_P/K_P^0)^{\Pi_P}$ are regular cyclic groups of order~$k$. Recall that 
$$K=\{(\sigma,\tau)\in K_H\times K_P:~(\sigma^{\pi_1})^\psi=\tau^{\pi_2}\},$$
where $\pi_1$ and $\pi_2$ are the canonical epimorphisms from $K_H$ to $K_H/K_H^0$ and from $K_P$ to $K_P/K_P^0$, respectively. By the definition, $K$ is a subdirect product of $K_H$ and $K_P$. If $(\sigma,\tau)\in K$, then
$$(\Delta\times \Delta^\varphi)^{(\sigma,\tau)}=\Delta^{\sigma^{\pi_1}}\times \Delta^{\varphi\circ\tau^{\pi_2}}=\Delta^{\sigma^{\pi_1}}\times \Delta^{\varphi\circ(\sigma^{\pi_1})^\psi}=\Delta^{\sigma^{\pi_1}}\times \Delta^{\sigma^{\pi_1}\circ\varphi}\subseteq X,$$
where the third equality holds by Eq.~\eqref{commutebij}. Therefore $X$ is $K$-invariant. Since $K_H$ and $K_P$ are transitive on $X_H$ and $X_P$, respectively, $\Pi_H=\orb(K_H^0,X_H)$, $\Pi_P=\orb(K_P^0,X_P)$, and $K$ is a subdirect product of $K_H$ and $K_P$, we conclude that $K$ is transitive on~$X$. Thus, $X\in \orb(K,G)$.
\end{proof}

\begin{rem}\label{subdirectrem}
It can be easily verified that $\psi$ satisfying Eq.~\eqref{commutebij} always exists in Lemma~\ref{subdirect} whenever $k\leq 3$. Indeed, in this case, any bijection from $K_H/K_H^0$ to $K_P/K_P^0$ mapping the trivial element of $K_H/K_H^0$ to the trivial element of $K_P/K_P^0$, in particular, conjugation by $\varphi$, is a group isomorphism.
\end{rem}

\section{$S$-rings over groups of order~$8$}

The next two lemmas follow from computer calculations using the package~\cite{GAP}. In these lemmas, $H\cong C_4\times C_2$ or $H\cong E_8$. In the former case, $A$ and $B$ are subgroups of $H$ such that $A\cong C_4$, $B\cong C_2$, and $H=A\times B$, $a$ and $b$ are generators of~$A$ and~$B$, respectively, $a_0=a^2$, and $A_0=\langle a_0\rangle$. In the latter one, $A$, $B$, and $C$ are subgroups of $H$ such that $A\cong B\cong C \cong C_2$ and $H=A\times B\times C$, and $a$, $b$, and~$c$ are generators of~$A$,~$B$, and~$C$, respectively.

\begin{lemm}\label{c4c2}
Let $\mathcal{A}$ be an $S$-ring over $H$. Then $\mathcal{A}$ is Cayley isomorphic up to taking the dual $S$-ring to one of the following $S$-rings:
\begin{enumerate}

\tm{1} $\mathcal{T}_H$;

\tm{2} a nontrivial tensor product of $S$-rings over groups of orders~$4$ and~$2$;

\tm{3} a nontrivial wreath product of an $S$-ring $\mathcal{B}$ over $L\lneq H$ and $\mathcal{T}_{H/L}$, where $H\cong C_4\times C_2$ and $L\in\{A_0,A\}$, or $H\cong E_8$ and $|L|=2$, or $L\cong E_4$ and $\mathcal{B}=\mathbb{Z}L$;

\tm{4} $\cyc(K,H)$, where $K\leq \aut(H)$ is one of the groups from Table~$1$.

\end{enumerate}

\end{lemm}

\begin{table}[h]

{\small
\begin{tabular}{|l|l|l|l|l|}
  \hline
  $H$ & $K$ & generators of~$K$ & structure of~$K$ & $\cyc(K,H)$ \\
  \hline
  $C_4\times C_2$ &$K_1$ & $(a,b)\mapsto (a,a_0b)$  & $C_2$ & $\mathbb{Z}C_4\wr_{C_4/C_2} \mathbb{Z}E_4$ \\ \hline
	$C_4\times C_2$ &$K_2$ & $(a,b)\mapsto (ab,a_0b),~(a,b)\mapsto (ab,b)$  & $D_8$ & $(\mathbb{Z}C_2\wr\mathbb{Z}C_2)\wr \mathbb{Z}C_2$ \\  \hline
	$E_8$ & $K_3$ & $(a,b,c)\mapsto (a,ab,bc),~(a,b,c)\mapsto (a,b,bc)$  & $D_8$ & $(\mathbb{Z}C_2\wr\mathbb{Z}C_2)\wr \mathbb{Z}C_2$ \\\hline
	
	\end{tabular}
}
\caption{$S$-rings over $C_4\times C_2$ and $E_8$}
\end{table}

\begin{rem}\label{cycle8}
In fact, all $S$-rings over $E_8$ are cyclotomic due to computer calculations using the package~\cite{GAP}.
\end{rem}

\begin{lemm}\label{8rank2}
Let $\mathcal{A}$ be an $S$-ring over~$H$ and $L$ an $\mathcal{A}$-subgroup of order~$4$. Suppose that $\mathcal{A}_L=\mathcal{T}_L$. Then $\mathcal{A}=\mathcal{A}_L\wr \mathcal{A}_{H/L}$ or $\mathcal{A}_L$ is $\otimes$-complemented in $\mathcal{A}$. 
\end{lemm}

\section{$S$-rings over $C_4\times C_{2p}$ and $E_8\times C_p$}

In this section, we provide a description of $S$-rings over the groups $C_4\times C_{2p}$ and $E_8\times C_p$, where $p$ is an odd prime. An $S$-ring $\mathcal{A}$ over one of the above groups is said to be \emph{dense} if the subgroups of orders $8$ and $p$ are $\mathcal{A}$-subgroups. The main result of the section is the theorem below. 

\begin{theo}\label{8p}
Let $p$ be an odd prime, $G\cong C_4\times C_{2p}$ or $G\cong E_8\times C_p$, and $\mathcal{A}$ a nontrivial $S$-ring over~$G$. Then $\mathcal{A}$ is cyclotomic, or a nontrivial tensor product, or a nontrivial $S$-wreath product for some $\mathcal{A}$-section $S=U/L$ and one of the following statements holds:

\begin{enumerate}

\tm{1} $|S|\leq 2$;

\tm{2} $|S|=4$ and $\mathcal{A}_{S}\neq \mathcal{T}_{S}$;

\tm{3} $\mathcal{A}_{L}$ is $\otimes$-complemented in $\mathcal{A}_{U}$ or $\mathcal{A}_{S}$ is $\otimes$-complemented in $\mathcal{A}_{G/L}$;

\tm{4} $\mathcal{A}$ is dense and $|U|=|G/L|=4p$.
\end{enumerate} 

\end{theo}

Throughout this section, $H$ is isomorphic to~$C_4\times C_2$ or $E_8$, $p$ is an odd prime, $P\cong C_p$, $G=H\times P$, and $\mathcal{A}$ is a nontrivial $S$-ring over $G$. We divide the proof of Theorem~\ref{8p} into two parts depending on whether $\mathcal{A}$ is dense or not.

\subsection{Nondense $S$-rings}

A description of nondense $S$-rings over $G$ is given in two propositions below. 

\begin{prop}\label{hnot8p}
Suppose that $H$ is not an $\mathcal{A}$-subgroup. Then $\mathcal{A}$ is a nontrivial tensor product or a nontrivial $S$-wreath product for some $\mathcal{A}$-section $S=U/L$ satisfying one of the following conditions:
\begin{enumerate}

\tm{1} $|S|\leq 2$;

\tm{2} $|S|=4$ and $\mathcal{A}_{S}\neq \mathcal{T}_{S}$;

\tm{3} $\mathcal{A}_{S}$ is $\otimes$-complemented in $\mathcal{A}_{G/L}$.

\end{enumerate}
\end{prop}

\begin{proof}
Let $H_1$ be a maximal $\mathcal{A}$-subgroup contained in $H$ and $P_1$ the least $\mathcal{A}$-subgroup containing $P$. Since $H$ is not an $\mathcal{A}$-subgroup, $H_1<H$ and hence one of the statements of Lemma~\aref{A-nonpowernew1} holds for $\mathcal{A}$. Suppose that Statement~$(1)$ of Lemma~\aref{A-nonpowernew1} holds for $\mathcal{A}$, i.e. $H_1P_1=G$, $P_1\lneq G$, and $\mathcal{A}=\mathcal{A}_{H_1} \star \mathcal{A}_{P_1}$. If $|H_1\cap P_1|=1$, then $\mathcal{A}=\mathcal{A}_{H_1}\otimes \mathcal{A}_{P_1}$, i.e. $\mathcal{A}$ is a nontrivial tensor product as required. If $|H_1\cap P_1|>1$, then $\mathcal{A}$ is the nontrivial $S=H_1/(H_1\cap P_1)$-wreath product. As $H_1<H$, we conclude that $|H_1|\leq 4$ and consequently $|S|=|H_1/(H_1\cap P_1)|\leq 2$. Thus, $S$ satisfies Condition~$(1)$ from Proposition~\ref{hnot8p} and we are done. 

Now suppose that Statement~$(2)$ of Lemma~\aref{A-nonpowernew1} holds for $\mathcal{A}$, i.e. $\mathcal{A}=\mathcal{A}_{H_1}\wr \mathcal{A}_{G/H_1}$ with $\mathcal{A}_{G/H_1}=\mathcal{T}_{H_1}$. If $H_1$ is trivial, then so is $\mathcal{A}$, a contradiction to the assumption that $\mathcal{A}$ is nontrivial. Otherwise, $\mathcal{A}$ is the nontrivial $S=H_1/H_1$-wreath product and $S$ obviously satisfies Condition~$(1)$ from Proposition~\ref{hnot8p} as desired.

Finally, suppose that Statement~$(3)$ of Lemma~\aref{A-nonpowernew1} holds for $\mathcal{A}$, i.e. $\mathcal{A}$ is the nontrivial $(H_1P_1)/P_1$-wreath product. One can see that $|(H_1P_1)/P_1|\leq 4$ because $|H_1|\leq 4$. If $|(H_1P_1)/P_1|\leq 2$ or $|(H_1P_1)/P_1|=4$ and $\mathcal{A}_{(H_1P_1)/P_1}\neq \mathcal{T}_{(H_1P_1)/P_1}$, then the section $S=(H_1P_1)/P_1$ satisfies Condition~$(2)$ from Proposition~\ref{hnot8p} and we are done.

Due to the above paragraph, we may assume further that $|(H_1P_1)/P_1|=4$ and $\mathcal{A}_{(H_1P_1)/P_1}=\mathcal{T}_{(H_1P_1)/P_1}$. This implies that $|H_1\cap P_1|=1$ and hence $P_1=P$, $|H_1|=4$, and $\mathcal{A}_{H_1}=\mathcal{T}_{H_1}$. From Lemma~\ref{8rank2} applied to $\mathcal{A}_{G/P}$ it follows that $\mathcal{A}_{G/P}\cong\mathcal{A}_{H_1}\wr \mathcal{A}_{H/H_1}$ or $\mathcal{A}_{(H_1P)/P}$ is $\otimes$-complemented in $\mathcal{A}_{G/P}$. In the former case, $\mathcal{A}=\mathcal{A}_{H_1P}\wr \mathcal{A}_{G/(H_1P)}$ and the section $S=(H_1P)/(H_1P)$ satisfies Condition~$(1)$ from Proposition~\ref{hnot8p}. In the latter one, the section $S=(H_1P)/P$ satisfies Condition~$(3)$ from Proposition~\ref{hnot8p} which completes the proof. 
\end{proof}

\begin{prop}\label{pnot8p}
Suppose that $P$ is not an $\mathcal{A}$-subgroup. Then $\mathcal{A}$ is a nontrivial tensor product or a nontrivial $S$-wreath product for some $\mathcal{A}$-section $S=U/L$ satisfying one of the following conditions:
\begin{enumerate}

\tm{1} $|S|\leq 2$;

\tm{2} $|S|=4$ and $\mathcal{A}_{S}\neq \mathcal{T}_{S}$;

\tm{3} $\mathcal{A}_{L}$ is $\otimes$-complemented in $\mathcal{A}_{U}$.

\end{enumerate}
\end{prop}

\begin{proof}
Let $\widehat{\mathcal{A}}$ be the $S$-ring dual to $\mathcal{A}$ over $\widehat{G}\cong G$. Since $P$ is not an $\mathcal{A}$-subgroup, the group $P^\bot\cong H$ is not an $\widehat{\mathcal{A}}$-subgroup by Lemma~\aref{A-dual}(1). So Proposition~\ref{hnot8p} holds for $\widehat{\mathcal{A}}$. If $\widehat{\mathcal{A}}$ is a nontrivial tensor product, then so is $\mathcal{A}$ by Lemma~\aref{A-dual}(3) as desired.

Let $\widehat{\mathcal{A}}$ be a nontrivial $\widehat{U}/\widehat{L}$-wreath product for some $\widehat{\mathcal{A}}$-section $\widehat{U}/\widehat{L}$ satisfying one of Conditions~$(1)$-$(3)$ from Proposition~\ref{hnot8p}. Then $\mathcal{A}$ is a nontrivial $U/L$-wreath product, where $L=\widehat{U}^\bot$ and $U=\widehat{L}^\bot$, by Lemma~\aref{A-dual}(4). Observe that $|U/L|=|\widehat{U}/\widehat{L}|$ by Lemma~\aref{A-dual}(1) and $\widehat{\mathcal{A}_{U/L}}=\widehat{\mathcal{A}}_{\widehat{U}/\widehat{L}}$ by Lemma~\aref{A-dual}(2). So if $\widehat{U}/\widehat{L}$ satisfies Condition~$(1)$ or~$(2)$ from Proposition~\ref{hnot8p}, then $U/L$ satisfies Condition~$(1)$ or~$(2)$ from Proposition~\ref{pnot8p}, respectively, and we are done.

Lemma~\aref{A-dual}(2) implies that $\widehat{\mathcal{A}_{U}}=\widehat{\mathcal{A}}_{\widehat{G}/\widehat{L}}$. Together with $\widehat{\mathcal{A}_{U/L}}=\widehat{\mathcal{A}}_{\widehat{U}/\widehat{L}}$ and Lemma~\aref{A-dual}(3), this yields that if $\widehat{U}/\widehat{L}$ satisfies Condition~$(3)$ from Proposition~\ref{hnot8p}, i.e. $\widehat{U}/\widehat{L}$ is $\otimes$-complemented in $\widehat{\mathcal{A}}_{\widehat{G}/\widehat{L}}$, then $\mathcal{A}_{L}$ is $\otimes$-complemented in $\mathcal{A}_{U}$ and hence $U/L$ satisfies Condition~$(3)$ from Proposition~\ref{pnot8p} as required.
\end{proof}

\subsection{Dense $S$-rings}

A description of dense $S$-rings over $G$ is provided by the next proposition.

\begin{prop}\label{hpasubgroups8p}
Suppose that $\mathcal{A}$ is dense. Then $\mathcal{A}$ is cyclotomic, or a nontrivial tensor product, or a nontrivial $S$-wreath product for some $\mathcal{A}$-section $S=U/L$ satisfying one of the following conditions:
\begin{enumerate}

\tm{1} $L$ is $\otimes$-complemented in~$U$ or $\mathcal{A}_{S}$ is $\otimes$-complemented in $\mathcal{A}_{G/L}$;

\tm{2} $|U|=|G/L|=4p$.

\end{enumerate}

\end{prop}

Before we give a proof of Proposition~\ref{hpasubgroups8p}, we provide several lemmas on a structure of a basic set of $\mathcal{A}$. Throughout this subsection, we assume that $H$ and $P$ are $\mathcal{A}$-subgroups and use the notation for the elements and subgroups of $H$ from Lemma~\ref{c4c2}. Given $X\in \mathcal{S}(\mathcal{A})$, put 
$$\lambda_X=|X\cap Hx|,~x\in X.$$ 
Due to Lemma~\aref{A-intersection}, the number $\lambda_X$ does not depend on~$x\in X$.

Recall that a subset of $G$ is said to be \emph{regular} if it consists of elements of the same order. In the next three lemmas, $X\in \mathcal{S}(\mathcal{A})_{G\setminus (H\cup P)}$. Lemma~\aref{A-tenspr}(1) implies that $X_H\in \mathcal{S}(\mathcal{A}_H)$ and $X_P\in \mathcal{S}(\mathcal{A}_P)$. Due to Lemma~\ref{basesethp}, there exist uniform partitions $\Pi_H$ and $\Pi_P$ of $X_H$ and $X_P$, respectively, and a bijection $\varphi$ from $\Pi_H$ to $\Pi_P$ such that Eq.~\eqref{xrepresent} holds for~$X$. Since $\Pi_H$ is uniform, $\lambda_X=|\Delta|$ for every $\Delta\in \Pi_H$ and $\lambda_X$ divides~$|X_H|$. By Lemma~\ref{orbitspip}, there exist groups $K_P^0\leq K_P\leq \aut(P)$ such that $\mathcal{A}_P=\cyc(K_P,P)$, $X_P\in \orb(K_P,P)$, $\Pi_P=\orb(K_P^0,X_P)$, and $|K_P/K_P^0|=|\Pi_P|$.

\begin{lemm}\label{nonreg8p}
If $X_H$ is nonregular, then $\Delta=\Delta^{-1}$ for every $\Delta\in \Pi_H$ and $\lambda_X>1$.
\end{lemm}

\begin{proof}
Note that $H\cong C_4\times C_2$ because $X_H$ is nonregular. Assume the contrary that $\Delta\neq \Delta^{-1}$ for some $\Delta\in \Pi_H$. Then there are $h\in \Delta$ and $g\in \Delta^{\varphi}$ such that $|h|=4$, $hg\in X$, and
\begin{equation}\label{hgx}
h^{-1}g\notin X. 
\end{equation}
Since $X_H$ is nonregular, there is $h^\prime\in H^\#$ such that $|h^\prime|=2$ and $h^\prime g^\prime\in X$ for some $g^\prime\in X_P$. As $p$ is odd, there exists a positive integer~$m$ such that $m\equiv 1\mod~p$ and $m\equiv 3\mod~4$. Lemma~\aref{A-burn} implies that $X^{(m)}\in \mathcal{S}(\mathcal{A})$. Observe that $(h^\prime g^\prime)^m=h^\prime g^\prime\in X^{(m)}\cap X$ because $|h^\prime|=2$ and $|g^\prime|=p$. So $X^{(m)}=X$. This implies that $h^{-1}g=(hg)^m\in X^{(m)}=X$, a contradiction to Eq.~\eqref{hgx}. If $\lambda_X=1$, then there is $\Delta\in \Pi_H$ consisting of an element of order~$4$, a contradiction to the first part of the lemma. 
\end{proof}

\begin{lemm}\label{rank28p}
Let $X_H=H^\#$. Then $X=X_H\times X_P$ or $H\cong E_8$ and $X\in \orb(K,G)$ for a subdirect product $K\leq \aut(G)$ of some $K_H\leq \aut(H)$ such that $\mathcal{T}_H=\cyc(K_H,H)$ and $K_P$.
\end{lemm}

\begin{proof}
Since $\Pi_H$ is uniform, $\lambda_X$ divides $|X_H|=|H^\#|=7$ and hence $\lambda_X=7$ or $\lambda_X=1$. In the former case, $\Pi_H=\{X_H\}$ and hence $X=X_H\times X_P$ as desired. Suppose that $\lambda_X=1$. Then $|\Pi_H|=|\Pi_P|=7$ and $\Pi_H$ consists of singletons. If $H\cong C_4\times C_2$, then we obtain a contradiction to Lemma~\ref{nonreg8p}. Further, we assume that $H\cong E_8$.

Lemma~\ref{basesethp} implies that 
$$X=\bigcup \limits_{h\in H^\#}\{h\}\times \{h\}^\varphi.$$
Put $\Lambda_h=\{h\}^\varphi$, $h\in H^\#$. Since $|\Pi_P|=7$, we conclude that $K_P/K_P^0\cong C_7$. Let $\tau_0$ be a generator of $K_P$ and $h_0\in H^\#$. Clearly, $\tau_0$ induces the permutation of $\Pi_P$. Put $\Lambda_i=\Lambda_{h_0}^{\tau_0^i}$ and $\{h_i\}=(\Lambda_i)^{\varphi^{-1}}$, $i\in \mathbb{Z}_7$. Due to Lemma~\aref{A-burn},
$$X^{\id_H\times \tau_0^i}\in \mathcal{S}(\mathcal{A})$$
for every $i\in \mathbb{Z}$. Moreover, $X^{\id_H\times \tau_0^{i+7}}=X^{\id_H\times \tau_0^{i}}$ for every $i\in \mathbb{Z}$ because $K_P/K_P^0\cong C_7$. Put
$$X_i=X^{\id_H\times \tau_0^i},~i\in \mathbb{Z}_7.$$
In this notation, $X_0=X$. By the definition, we have
$$X_i=\bigcup \limits_{j\in \mathbb{Z}_7}\{h_j\}\times \Lambda_{j+i},~i\in \mathbb{Z}_7.$$ 
The set of elements from $H^\#$ which enter the element $\underline{X_i}\cdot \underline{X}^{-1}$ is $\{h_jh_{j+i}:~j\in \mathbb{Z}_7\}$. As $H^\#\in \mathcal{S}(\mathcal{A}_H)$, we obtain
$$\{h_jh_{j+i}:~j\in \mathbb{Z}_7\}=H^\#$$
for every $i\in \mathbb{Z}_7$. Using these equalities, one can deduce three possible Cayley tables for $H$ whose rows and columns are indexed by the elements $\{e\}\cup \{h_i:~i\in \mathbb{Z}_7\}$. In all of the cases, the bijection $\sigma_0\in \sym(H)$ such that 
$$e^{\sigma_0}=e,~h_i^{\sigma_0}=h_{i+1},~i\in \mathbb{Z}_7,$$
or equivalently,
$$e^{\sigma_0}=e,~\{h\}^{\sigma_0}=(\Lambda_h^{\tau_0})^{\varphi^{-1}}=\{h\}^{\varphi\tau_0\varphi^{-1}},~h\in H^\#$$
is an automorphism of $H$ of order~$7$. The latter equality implies that Eq.~\eqref{commutebij} holds for every $\sigma\in K_H=\langle \sigma_0\rangle$ and the isomorphism $\psi$ from $K_H=\langle \sigma_0 \rangle\cong C_7$ to $K_P/K_P^0=\langle K_P^0\tau_0 \rangle\cong C_7$ such that $\sigma_0^\psi=K_P^0\tau_0$. Thus, $X\in \orb(K,G)$, where $K=K(K_H,K_H^0,K_P,K_P^0,\psi)$ and $K_H^0$ is trivial, by Lemma~\ref{subdirect} as required.
\end{proof}

\begin{lemm}\label{wreath8p}
Let $X_H=H\setminus L$ for some nontrivial $\mathcal{A}_H$-subgroup $L\lneq H$ such that one of the following conditions holds:
\begin{enumerate}

\tm{1} $H\cong C_4\times C_2$ and $L\in\{A_0,A\}$;

\tm{2} $H\cong E_8$ and $|L|=2$;

\tm{3} $L\cong E_4$ and $\mathcal{A}_L=\mathbb{Z}L$.
\end{enumerate}
Then $L_0\leq \rad(X)$ for some $L_0\leq L$ such that $|L_0|=2$. 
\end{lemm}

\begin{proof}
If $\lambda_X=|X_H|$, then $X=X_H\times X_P$. Therefore $L\leq\rad(X_H)\leq \rad(X)$ and hence $L_0\leq \rad(X)$ for every subgroup $L_0$ of $L$ of order~$2$ as required. Further, we assume that
\begin{equation}\label{lambdaless}
\lambda_X<|X_H|.
\end{equation}

The assumption of the lemma implies that $|L|\in \{2,4\}$ and hence $|X_H|\in \{4,6\}$. We divide the proof into two cases depending on~$|X_H|$.

\hspace{5mm}

\noindent \textbf{Case~1: $|X_H|=4$.} In this case, $|L|=4$ and $X_H$ is an $L$-coset. Recall that $\lambda_X$ divides $|X_H|=4$. Together with Eq.~\eqref{lambdaless}, this yields that 
$$\lambda_X\in\{1,2\}.$$

If $L$ is cyclic, then $H\cong C_4\times C_2$, $L=A\cong C_4$, i.e. Condition~$(1)$ of the lemma holds, and $X_H$ is nonregular. Lemma~\ref{nonreg8p} implies that $\lambda_X>1$ and hence $\lambda_X=2$. It suffices to show that $A_0\leq \rad(X)$. Assume the contrary that $A_0\nleq \rad(X)$. Then $A_0\nleq \rad(\Delta)$ for some $\Delta\in \Pi_H$. Clearly, $|\Delta|=\lambda_X=2$. Since $h^{-1}=a_0h$ for every $h\in H$ with $|h|=4$, where $a_0$ is the nontrivial element of~$A_0$, we conclude that $\Delta$ is nonregular and $\Delta\neq \Delta^{-1}$, a contradiction to Lemma~\ref{nonreg8p}. 

Let $L$ be a noncyclic. Then $L\cong E_4$. So Condition~$(3)$ of the lemma holds and hence $\mathcal{A}_L=\mathbb{Z}L$. If $\lambda_X=2$, then $\Pi_H$ is a partition of $X_H=H\setminus L$ into two $L_0$-cosets for some subgroup $L_0$ of $L$ of order~$2$ because $L\cong E_4$ and $X_H$ is an $L$-coset. This yields that $L_0\leq \rad(\Delta)$ for every $\Delta\in \Pi_H$ and hence $L_0\leq \rad(X)$ as desired.

Suppose that $\lambda_X=1$. Then $\Pi_H$ consists of four singletons and hence 
$$X=\bigcup \limits_{h\in X_H} \{h\}\times \Lambda_h,$$
where $\{\Lambda_h:~h\in X_H\}=\Pi_P$, by Lemma~\ref{basesethp}. Let $L_1$ and $L_2$ be subgroups of $L$ of order~$2$ such that $L=L_1\times L_2$, $h_1$ and $h_2$ be the nontrivial elements of $L_1$ and $L_2$, respectively, $\pi_1$ and $\pi_2$ be the canonical epimorphisms from $G$ to $G/L_1$ and from $G$ to $G/L_2$, respectively, and $h\in H\setminus L$. Note that $L_1$ and $L_2$ are $\mathcal{A}$-subgroups because $\mathcal{A}_L=\mathbb{Z}L$. The sets
$$X^{\pi_1}=h^{\pi_1}\times (\Lambda_h\cup \Lambda_{hh_1})^{\pi_1}\cup (hh_2)^{\pi_1}\times (\Lambda_{hh_2}\cup \Lambda_{hh_1h_2})^{\pi_1}$$
and
$$X^{\pi_2}=h^{\pi_2}\times (\Lambda_h\cup \Lambda_{hh_2})^{\pi_2}\cup (hh_1)^{\pi_2}\times (\Lambda_{hh_1}\cup \Lambda_{hh_1h_2})^{\pi_2}$$
are basic sets of $\mathcal{A}_{G/L_1}$ and $\mathcal{A}_{G/L_2}$, respectively. Due to Lemma~\ref{orbitspip} applied to $X^{\pi_1}$ and $X^{\pi_2}$, we have 
$$\{\Lambda_h\cup \Lambda_{hh_1},\Lambda_{hh_2}\cup \Lambda_{hh_1h_2}\}=\orb(\widetilde{K}_P^0,X_P)=\{\Lambda_h\cup \Lambda_{hh_2},\Lambda_{hh_1}\cup \Lambda_{hh_1h_2}\},$$
where $\widetilde{K}_P^0$ is the subgroup of $K_P$ of index~$2$. The latter equalities contradict to the fact that $\Pi_P$ is a partition of $X_P$.

\hspace{5mm}

\noindent \textbf{Case~2: $|X_H|=6$.} In this case, $|L|=2$ and $X_H$ is a union of three $L$-cosets. Let $u$ be the nontrivial element of~$L$. As $\lambda_X$ divides $|X_H|=6$ and Eq.~\eqref{lambdaless} holds, 
$$\lambda_X\in\{1,2,3\}.$$

Suppose that $\lambda_X=3$. Then $|\Pi_H|=|\Pi_P|=2$ and Lemma~\ref{basesethp} implies that
$$X=\Delta_0\times \Lambda_0\cup \Delta_1\times \Lambda_1,$$
where $\{\Delta_0,\Delta_1\}=\Pi_H$, $|\Delta_0|=|\Delta_1|=\lambda_X=3$, and $\{\Lambda_0,\Lambda_1\}=\Pi_P$. If $H\cong C_4\times C_2$, then $L=A_0$ and $u=a_0$, i.e. Condition~$(1)$ of the lemma holds. In this case, $X_H$ is nonregular. Observe that $\Delta_0$ contains an element~$h$ of order~$4$ and $A_0\nleq \rad(\Delta_0)$ because $|\Delta_0|=3$. Therefore $h^{-1}=ha_0\notin\Delta_0$ and consequently $\Delta_0\neq \Delta_0^{-1}$, a contradiction to Lemma~\ref{nonreg8p}	

Let $H\cong E_8$. Again, $L\nleq \rad(\Delta_0)$ because $|\Delta_0|=3$. Since $\Delta_0\cup \Delta_1=H\setminus L$, this implies that $\Delta_1=u\Delta_0$. One can see that 
$$X^{-1}=\Delta_0\times \Lambda_0^{-1}\cup \Delta_1\times \Lambda_1^{-1}$$
because $\Delta_0$ and $\Delta_1$ consist of elements of order~$2$. Let $\Delta_0=\{h_1,h_2,h_3\}$. As $\Delta_1=u\Delta_0$, the set of elements from $H^\#$ which enter the element $\underline{X}\cdot \underline{X}^{-1}$ coincides with the set of elements from $H^\#$ which enter the element $\underline{\Delta_0}^2$, i.e.
$$\{h_1h_2,h_1h_3,h_2h_3\}.$$
The above set is a nontrivial $\mathcal{A}_H$-set of size~$3$. However, the nontrivial basic sets of $\mathcal{A}_H$ are $L^\#=\{u\}$ and $X_H=H\setminus L$ of sizes~$1$ and~$6$ and hence there is no an $\mathcal{A}_H$-set of size~$3$, a contradiction.

Suppose that $\lambda_X=2$. Then $|\Pi_H|=|\Pi_P|=3$ and $|\Delta|=2$ for every $\Delta\in \Pi_H$. We are done if $L\leq \rad(\Delta)$ for every $\Delta\in \Pi_H$. Assume the contrary that $L\nleq \rad(\Delta)$ for some $\Delta\in \Pi_H$. Then $\mu_X=|H^\pi x^\pi\cap X^\pi|=2$ for every $x\in X$, where $\pi$ is the canonical epimorphism from $G$ to $G/L$. However, Lemma~\ref{basesethp} applied to $X^\pi$ yields that $\mu_X$ divides $|X^\pi_{H^\pi}|=|(X_H)^\pi|=3$, a contradiction.

Finally, suppose that $\lambda_X=1$. Due to Lemma~\ref{nonreg8p}, we have $H\cong E_8$. Note that $|\Pi_H|=|\Pi_P|=6$ and $\Pi_H$ consists of singletons. Lemma~\ref{basesethp} implies that 
$$X=\bigcup \limits_{h\in X_H}\{h\}\times \{h\}^\varphi.$$
Put $\Lambda_h=\{h\}^\varphi$, $h\in H^\#$. As $|\Pi_P|=6$, we conclude that $K_P/K_P^0\cong C_6$. Let $\tau_0$ be a generator of $K_P$ and $h_0\in H^\#$. Clearly, $\tau_0$ induces the permutation of $\Pi_P$. Put $\Lambda_i=\Lambda_{h_0}^{\tau_0^i}$ and $\{h_i\}=(\Lambda_i)^{\varphi^{-1}}$, $i\in \mathbb{Z}_6$. In this notation,
$$X=\bigcup \limits_{i\in \mathbb{Z}_6} \{h_i\}\times \Lambda_i.$$

Given $i\in \mathbb{Z}_6$, let $i^\prime\in \mathbb{Z}_6$ be such that $h_{i^\prime}=h_iu$. Let $\pi$ be the canonical epimorphism from $G$ to $G/L$. Then 
$$X^\pi=\bigcup \limits_{i\in \mathbb{Z}_6} h_i^\pi\times (\Lambda_i\cup \Lambda_{i^\prime}).$$
Lemma~\ref{orbitspip} applied to $X^\pi$ yields that $\{\Lambda_i\cup \Lambda_{i^\prime}:~i\in \mathbb{Z}_6\}$ is the set of orbits of the subgroup $\langle \tau_0^2\rangle$ of $K_P$ of index~$3$ on $X_P$. Therefore $i^\prime=i+3$ and hence
\begin{equation}\label{shift2}
h_iu=h_{i+3},~i\in \mathbb{Z}_6.
\end{equation}

Due to Lemma~\aref{A-burn},
$$X^\prime=X^{\id_H\times \tau_0}=\bigcup \limits_{i\in \mathbb{Z}_6}\{h_i\}\times \Lambda_{i+1}\in \mathcal{S}(\mathcal{A}).$$
The set of elements from $H^\#$ which enter the element $\underline{X}^\prime\cdot \underline{X}^{-1}$ is
$$\Gamma=\{h_ih_{i+1}:~i\in \mathbb{Z}_6\}.$$
Clearly, $\Gamma$ is an $\mathcal{A}_H$-set. An explicit computation using Eq.~\eqref{shift2} implies that 
$$\Gamma=\{h_0h_1,h_1h_2,h_0h_2u\}.$$
In particular, $|\Gamma|=3$. Again, we obtain a contradiction to the fact that $L^\#=\{u\}$ and $X_H=H\setminus L$ are the only nontrivial basic sets of $\mathcal{A}_H$. 
\end{proof}

\begin{proof}[Proof of Proposition~\ref{hpasubgroups8p}]
The $S$-ring $\mathcal{A}_H$ is Cayley isomorphic to one of the $S$-rings from Statements~$(1)$-$(4)$ of Lemma~\ref{c4c2} or to an $S$-ring dual to one of them.

\begin{lemm}\label{reduction8p}
If Proposition~\ref{hpasubgroups8p} holds in case when $\mathcal{A}_H$ is Cayley isomorphic to one of the $S$-rings from Statements~$(1)$-$(4)$ of Lemma~\ref{c4c2}, then it also holds in case when $\mathcal{A}_H$ is Cayley isomorphic to an $S$-ring dual to one of them.
\end{lemm}

\begin{proof}
From Lemma~\aref{A-dual}(1) it follows that the subgroups $P^\bot$ and $H^\bot$ of $\widehat{G}$ of orders~$8$ and~$p$, respectively, are $\widehat{\mathcal{A}}$-subgroups. Suppose that $\mathcal{A}_H$ is Cayley isomorphic to an $S$-ring dual to one of the $S$-rings from Statements~$(1)$-$(4)$ of Lemma~\ref{c4c2}. Then $\widehat{\mathcal{A}}_{P^\bot}$ is Cayley isomorphic to one of the above $S$-rings by Lemma~\aref{A-dual}(2). So Proposition~\ref{hpasubgroups8p} holds for $\widehat{\mathcal{A}}$ by the assumption of the lemma. Now the required follows from the fact that the $S$-ring dual to $\widehat{\mathcal{A}}$ is equal to~$\mathcal{A}$ and Statements~$(2)$-$(5)$ of Lemma~\aref{A-dual}.
\end{proof}

In view of Lemma~\ref{reduction8p}, we may assume that $\mathcal{A}_H$ is one of the $S$-rings from Statements~$(1)$-$(4)$ of Lemma~\ref{c4c2}. Let us consider all the cases. Suppose that Statement~$(1)$ of Lemma~\ref{c4c2} holds for $\mathcal{A}_H$, i.e. $\mathcal{A}_H=\mathcal{T}_H$. Let $X\in \mathcal{S}(\mathcal{A})_{G\setminus (H\cup P)}$. Then $X_H=H^\#$ by Lemma~\aref{A-tenspr}(1). Lemma~\ref{rank28p} implies that $X=X_H\times X_P$ or $H\cong E_8$ and $X\in \orb(K,G)$ for a subdirect product $K\leq \aut(G)$ of some $K_H\leq \aut(H)$ such that $\mathcal{T}_H=\cyc(K_H,H)$ and some $K_P\leq \aut(P)$ such that $\mathcal{A}_P=\cyc(K_P,P)$. Due to Lemma~\ref{conj}, every $Y\in \mathcal{S}(\mathcal{A})_{G\setminus (H\cup P)}$ is rationally conjugate to~$X$. Therefore $Y=Y_H\times Y_P$ for every $Y\in \mathcal{S}(\mathcal{A})_{G\setminus (H\cup P)}$ or $Y\in \orb(K,G)$ for every $Y\in \mathcal{S}(\mathcal{A})_{G\setminus (H\cup P)}$. In the former case, $\mathcal{A}=\mathcal{A}_H\otimes \mathcal{A}_P$, whereas in the latter one, $\mathcal{A}=\cyc(K,G)$. In both cases, the conclusion of Proposition~\ref{hpasubgroups8p} holds and we are done.

Suppose that Statement~$(2)$ of Lemma~\ref{c4c2} holds for $\mathcal{A}_H$, i.e. $\mathcal{A}_H=\mathcal{A}_U\otimes \mathbb{Z}L$ for $\mathcal{A}_H$-subgroups $U$ and $L$ such that $|U|=4$, $|L|=2$, and $H=U\times L$. Then $\mathcal{A}=\mathcal{A}_{U\times P}\otimes \mathbb{Z}L$ by Lemma~\aref{A-tenspr}(2) and we are done.

Suppose that Statement~$(3)$ of Lemma~\ref{c4c2} holds for $\mathcal{A}_H$. In this case, $\mathcal{A}_H=\mathcal{A}_L\wr \mathcal{T}_{H/L}$ for some proper nontrivial $\mathcal{A}_H$-subgroup $L$ of $H$ and one of the additional conditions given in Statement~$(2)$ of Lemma~\ref{c4c2} holds. Let $X\in \mathcal{S}(\mathcal{A})_{G\setminus ((L\times P)\cup H)}$. Then $L_0\leq \rad(X)$ by Lemma~\ref{wreath8p} for some subgroup $L_0$ of $L$ of order~$2$. Since $\mathcal{A}_{H/L}=\mathcal{T}_{H/L}$, every $Y\in \mathcal{S}(\mathcal{A})_{G\setminus ((L\times P)\cup H)}$ is rationally conjugate to~$X$ by Lemma~\ref{conj}. Therefore $L_0\leq \rad(Y)$ for every $Y\in \mathcal{S}(\mathcal{A})_{G\setminus ((L\times P)\cup H)}$. Thus, $\mathcal{A}$ is the $S=(L\times P)/L_1$-wreath product, where $L_1\leq L$ is the least $\mathcal{A}$-subgroup containing $L_0$.

Let $L_1=L$. If $|L|=2$, then $\mathcal{A}_L=\mathbb{Z}L$. Lemma~\aref{A-tenspr}(2) implies that $\mathcal{A}_{L\times P}=\mathbb{Z}L\otimes \mathcal{A}_P$. So $S$ satisfies Condition~$(1)$ from Proposition~\ref{hpasubgroups8p}. If $|L|=4$, then $S\cong P$ and $H/L$ is an $\mathcal{A}_{G/L}$-subgroup of order~$2$. So $\mathcal{A}_{H/L}=\mathbb{Z}(H/L)$ and hence $\mathcal{A}_{G/L}=\mathcal{A}_S\otimes \mathbb{Z}(H/L)$ by Lemma~\aref{A-tenspr}(2). This yields that $S$ satisfies Condition~$(1)$ from Proposition~\ref{hpasubgroups8p}.

Let $L_1<L$, then $L_1=L_0$, $|L_1|=2$, and $|L|=4$. Therefore $|L\times P|=|G/L_1|=4p$ and $S$ satisfies Condition~$(2)$ from Proposition~\ref{hpasubgroups8p}.

It remains to consider the case when Statement~$(4)$ of Lemma~\ref{c4c2} holds for $\mathcal{A}_H$, i.e. $\mathcal{A}_H=\cyc(K_H,H)$, where $K_H\leq \aut(H)$ is one of the groups $K_1$-$K_3$ from Table~$1$. In all cases, there is an $\mathcal{A}_H$-subgroup $U$ of order~$4$. Let 
$$X\in \mathcal{S}(\mathcal{A})_{G\setminus ((U\times P)\cup H)}.$$
Then $X_H\in \mathcal{S}(\mathcal{A}_H)$ by Lemma~\aref{A-tenspr}(1) and hence 
$$X_H\in \orb(K_H,H).$$
Lemma~\ref{basesethp} yields that there exist uniform partitions $\Pi_H$ and $\Pi_P$ of $X_H$ and $X_P$, respectively, and a bijection $\varphi$ from $\Pi_H$ to $\Pi_P$ such that Eq.~\eqref{xrepresent} holds. Lemma~\aref{A-intersection} implies that the number $\lambda_X=|Hx\cap X|$ does not depend on $x\in X$. Since $\Pi_H$ is uniform, $\lambda_X$ divides $|X_H|$. Due to Lemma~\ref{orbitspip}, there exist groups $K_P^0\leq K_P\leq \aut(P)$ such that $\mathcal{A}_P=\cyc(K_P,P)$, $X_P\in \orb(K_P,P)$, $\Pi_P=\orb(K_P^0,X_P)$, and $|K_P/K_P^0|=|\Pi_P|$. 

Let $K_H=K_1$. Then $H\cong C_4\times C_2$, $U=A$, $\mathcal{A}_A=\mathbb{Z}A$, $\mathcal{S}(\mathcal{A}_H)_{H\setminus A}=\{A_0b,A_0ab\}$, and $|K_H|=|X_H|=2$. At first, let us prove that 
$$X\in \orb(K,G)$$
for some subdirect product $K\leq \aut(G)$ of $K_H$ and $K_P$. As $\lambda_X$ divides $|X_H|=2$, we have $\lambda_X\in\{1,2\}$. If $\lambda_X=2$, then $X=X_H\times X_P$ and hence $X\in \orb(K,G)$ for $K=K_H\times K_P$. If $\lambda_X=1$, then $|\Pi_H|=|\Pi_P|=2$ and $\Pi_H$ consists of two singletons. So $\Pi_H=\orb(K_H^0,X_H)$, where $K_H^0$ is the trivial subgroup of $K_H$, and $|K_P:K_P^0|=2$. One can see that $K_H/K_H^0\cong K_P/K_P^0\cong C_2$. Therefore $X\in \orb(K,G)$, where $K=K(K_H,K_H^0,K_P,K_P^0,\psi)$ for a unique isomorphism $\psi$ from $K_H/K_H^0\cong C_2$ to $K_P/K_P^0\cong C_2$, by Lemma~\ref{subdirect} and Remark~\ref{subdirectrem}.

Now let us prove that
\begin{equation}\label{orbt}
T\in \orb(K,G)
\end{equation}
for every $T\in \mathcal{S}(\mathcal{A})$ which is equivalent to $\mathcal{A}=\cyc(K,G)$ as desired. If $T\in \mathcal{S}(\mathcal{A}_H)\cup \mathcal{S}(\mathcal{A}_P)$, then Eq.~\eqref{orbt} holds because $K$ is a subdirect product of $K_H$ and $K_P$. If $T\in \mathcal{S}(\mathcal{A})_{(A\times P)\setminus (A\cup P)}$, then $T=\{h\}\times T_P$, where $h\in A^\#$ and $T_P\in \mathcal{S}(\mathcal{A}_P)$, by Lemma~\aref{A-tenspr}(2). Together with the definition of $K_H=K_1$, this yields that Eq.~\eqref{orbt} holds. Finally, suppose that $T\in \mathcal{S}(\mathcal{A})_{G\setminus ((A\times P)\cup H)}$. Then $T_H=hX_H$ for some $h\in \{e,a\}$ because $\mathcal{S}(\mathcal{A}_H)_{H\setminus A}=\{A_0b,A_0ab\}$. Together with Lemma~\ref{conj}, this implies that $T=hX^{(m)}$ for some $h\in \{e,a\}$ and integer $m$ coprime to~$2p$ and hence Eq.~\eqref{orbt} holds.

Let $K_H\in\{K_2,K_3\}$. Then $\mathcal{A}_{U}=\mathbb{Z}L\wr \mathbb{Z}(U/L)\cong \mathbb{Z}C_2\wr \mathbb{Z}C_2$ for an $\mathcal{A}_U$-subgroup $L$ of order~$2$,
$$\mathcal{S}(\mathcal{A}_H)=\{\{e\},\{h_0\},Lh_1,Uh_2\},$$
where $h_0$ is the nontrivial element of $L$, $h_1\in U\setminus L$, and $h_2\in H\setminus U$,
$K_H=\langle \sigma_1 \rangle \rtimes \langle \sigma_2 \rangle\cong D_8$, where $\sigma_1,\sigma_2\in \aut(H)$ are such that
$$h_0^{\sigma_1}=h_0^{\sigma_2}=h_0,~h_1^{\sigma_1}=h_0h_1,~h_1^{\sigma_2}=h_1,~h_2^{\sigma_1}=h_2^{\sigma_2}=h_1h_2,$$
and 
$$X_H=Uh_2.$$
The latter implies that 
$$\lambda_X\in \{1,2,4\}.$$
Let $Y\in \mathcal{S}(\mathcal{A})_{Lh_1P^\#}$ such that $Y_P=X_P$. Clearly,
$$Y_H\in \orb(K_H,H)$$
and
$$\lambda_Y=|Y\cap Hy|\in\{1,2\}.$$
Lemma~\ref{conj} implies that every basic set from $\mathcal{S}(\mathcal{A})_{G\setminus (H\cup (L\times P))}$ is rationally conjugate to $X$ or $Y$, whereas every basic set $T\in\mathcal{S}(\mathcal{A})_{L^\#\times P}$ is of the form $T=\{h_0\}\times T_P$ by Lemma~\aref{A-tenspr}(2). Thus, to prove that $\mathcal{A}$ is cyclotomic, it suffices to find a subdirect product $K$ of $K_H$ and $K_P$ such that 
$$X,Y\in \orb(K,G)$$ 
which will be done further. We divide the rest of the proof into several cases depending on~$\lambda_X$.

\hspace{5mm}

\noindent \textbf{Case~$1$: $\lambda_X=4$.} In this case, $X=X_H\times X_P$, $\Pi_H=\{X_H\}$, and $\Pi_P=\{X_P\}$. If $\lambda_Y=2$, then $Y=Y_H\times Y_P$ and hence $X,Y\in \orb(K,G)$ for $K=K_H\times K_P$. Suppose that $\lambda_Y=1$. Then Lemma~\ref{basesethp} and Lemma~\ref{orbitspip} imply that $K_P^0$ is the subgroup of $K_P$ of index~$2$ and
\begin{equation}\label{yform}
Y=\{h_1\}\times \Gamma_{h_1} \cup \{h_0h_1\} \times \Gamma_{h_0h_1},
\end{equation}
where $\{\Gamma_{h_1},\Gamma_{h_0h_1}\}=\orb(K_P^0,Y_P)$. Let $K_H^0=\langle \sigma_1^2,\sigma_2\rangle \cong E_4$ and $\psi$ a unique isomorphism from $K_H/K_H^0\cong C_2$ to $K_P/K_P^0\cong C_2$. By the definition of $K_H^0$, we have $\Pi_H=\{X_H\}=\orb(K_H^0,X_H)$ and $\{\{h_1\},\{h_0h_1\}\}=\orb(K_H^0,Y_H)$. Together with Lemma~\ref{subdirect} and Remark~\ref{subdirectrem}, the latter yields that $Y\in \orb(K,G)$, where $K=K(K_H,K_H^0,K_P,K_P^0,\psi)$. As $\Pi_H=\{X_H\}=\orb(K_H^0,X_H)$, we conclude that $X=X_H\times X_P\in \orb(K,G)$ as desired.

\hspace{5mm}

\noindent \textbf{Case~$2$: $\lambda_X=2$.} In this case, $K_P^0$ is the subgroup of $K_P$ of index~$2$ and $\Pi_H$ coincides with one of the partitions 
$$\Pi_1=\{\{h_2,h_0h_2\},\{h_1h_2,h_0h_1h_2\}\},~\Pi_2=\{\{h_2,h_1h_2\},\{h_0h_2,h_0h_1h_2\}\},$$
$$\Pi_3=\{\{h_2,h_0h_1h_2\},\{h_0h_2,h_1h_2\}\}.$$
In all of these cases, $\rad(\Delta)$ is a subgroup of $U$ of order~$2$ not depending on $\Delta\in \Pi_H$. Since $\mathcal{A}_{U}=\mathbb{Z}L\wr \mathbb{Z}(U/L)$, the subgroup $L$ is a unique $\mathcal{A}_U$-subgroup of order~$2$. We conclude that $\Pi_H$ is a partition of $X_H=Uh_2$ into the $L$-cosets, i.e.
\begin{equation}\label{piha0}
\Pi_H=\{Lh_2,Lh_1h_2\}.
\end{equation}

If $\lambda_Y=2$, i.e. $Y=Y_H\times Y_P$, then put $K_H^0=\langle \sigma_1^2,\sigma_2\sigma_1\rangle \cong E_4$. It is easy to verify using Eq.~\eqref{piha0} that $\Pi_H=\orb(K_H^0,X_H)$. So Lemma~\ref{subdirect} and Remark~\ref{subdirectrem} yield that $X\in \orb(K,G)$, where $K=K(K_H,K_H^0,K_P,K_P^0,\psi)$ and $\psi$ is a unique isomorphism from $K_H/K_H^0\cong C_2$ to $K_P/K_P^0\cong C_2$. One can see also that $\{Y_H\}=\{\{h_1,h_0h_1\}\}=\orb(K_H^0,Y_H)$ and hence $Y=Y_H\times Y_P\in \orb(K,G)$.

If $\lambda_Y=1$, i.e. $Y$ is of the form from Eq.~\eqref{yform}, then put $\widetilde{K}_H=\langle \sigma_1 \rangle\cong C_4$ and $\widetilde{K}_H^0=\langle \sigma_1^2\rangle\cong C_2$. It is easy to see that $X_H,Y_H\in \orb(\widetilde{K}_H,H)$. Using Eqs.~\eqref{yform} and~\eqref{piha0}, one can check that $\Pi_H=\orb(\widetilde{K}_H^0,X_H)$ and $\{\{h_1\},\{h_0h_1\}\}=\orb(\widetilde{K}_H^0,Y_H)$. Therefore $X,Y\in \orb(K,G)$, where $K=K(\widetilde{K}_H,\widetilde{K}_H^0,K_P,K_P^0,\psi)$ and $\psi$ is a unique isomorphism from $\widetilde{K}_H/\widetilde{K}_H^0\cong C_2$ to $K_P/K_P^0\cong C_2$ by Lemma~\ref{subdirect} and Remark~\ref{subdirectrem}.

\hspace{5mm}

\noindent \textbf{Case~$3$: $\lambda_X=1$.} In this case, $K_P^0$ is the subgroup of $K_P$ of index~$4$ by Lemma~\ref{orbitspip}, $\Pi_H$ consists of four singletons and hence 
$$X=\bigcup \limits_{h\in X_H} \{h\}\times \Lambda_h,$$
where $\Lambda_h=\{h\}^\varphi$, where $\varphi$ is a bijection from $\Pi_H$ to $\Pi_P$ such that Eq.~\eqref{xrepresent} holds, and $\{\Lambda_h:~h\in X_H\}=\Pi_P$. Let $\widetilde{K}_H=\langle \sigma_1 \rangle \cong C_4$ and $\widetilde{K}_H^0$ the trivial subgroup of $\widetilde{K}_H$. One can see that $X_H,Y_H\in \orb(\widetilde{K}_H,H)$ and $\Pi_H=\orb(\widetilde{K}_H^0,X_H)$.

Let $\pi$ be the canonical epimorphism from $G$ to $G/L$. Then the set
$$X^{\pi}=h_2^{\pi}\times (\Lambda_{h_2}\cup \Lambda_{h_0h_2})^{\pi}\cup (h_1h_2)^{\pi}\times (\Lambda_{h_1h_2}\cup \Lambda_{h_0h_1h_2})^{\pi}$$
is a basic set of $\mathcal{A}_{G/L}$. From Lemma~\ref{orbitspip} applied to $X^{\pi}$ it follows that
\begin{equation}\label{orbitsind2} 
\{\Lambda_{h_2}\cup \Lambda_{h_0h_2},\Lambda_{h_1h_2}\cup \Lambda_{h_0h_1h_2}\}=\orb(\widetilde{K}_P^0,X_P),
\end{equation}
where $\widetilde{K}_P^0$ is the subgroup of $K_P$ of index~$2$. Since $K_P/K_P^0$ acts on $\Pi_P$ as a regular cyclic group of order~$4$, there exists a unique $\tau_1\in K_P/K_P^0$ such that $\Lambda_{h_2}^{\tau_1}=\Lambda_{h_1h_2}=\Lambda_{h_2^{\sigma_1}}$. In view of Eq.~\eqref{orbitsind2}, we conclude that $K_P/K_P^0=\langle \tau_1 \rangle$ and $\Lambda_{h_2}^{\tau_1^2}=\Lambda_{h_0h_2}=\Lambda_{h_2^{\sigma_1^2}}$. Therefore
\begin{equation}\label{cycl4}
\Lambda_h^{\tau}=\Lambda_{h^\sigma}
\end{equation}
for every $\sigma\in\widetilde{K}_H$ and every $\tau\in K_P/K_P^0$. Let $\psi$ be an isomorphism from $\widetilde{K}_H=\widetilde{K}_H/\widetilde{K}_H^0\cong C_4$ to $K_P/K_P^0\cong C_4$ such that 
$$\sigma_1^\psi=\tau_1$$ 
and $K=K(\widetilde{K}_H,\widetilde{K}_H^0,K_P,K_P^0,\psi)$. It follows from Eq.~\eqref{cycl4} that Eq.~\eqref{commutebij} holds for $\psi$, $\varphi$, every $\sigma\in \widetilde{K}_H$, and every $\tau\in K_P/K_P^0$. Therefore $X\in \orb(K,G)$ by Lemma~\ref{subdirect}. 

Let $\lambda_Y=1$, i.e. $Y$ is of the form from Eq.~\eqref{yform}. Then 
$$\{\Gamma_{h_1},\Gamma_{h_0h_1}\}=\orb(\widetilde{K}_P^0,Y_P)=\{\Lambda_{h_2}\cup \Lambda_{h_0h_2},\Lambda_{h_1h_2}\cup \Lambda_{h_0h_1h_2}\}$$
by Lemma~\ref{orbitspip}. Together with Eq.~\eqref{cycl4}, the above equality implies that $\Gamma_{h_1}^{\tau_1}=\Gamma_{h_0h_1}$ and $\Gamma_{h_1}^{\tau_1^2}=\Gamma_{h_1}$. Since also $h_1^{\sigma_1}=h_0h_1$ and $h_1^{\sigma_1^2}=h_1$, we conclude that $Y\in \orb(K,G)$ as required.

Let $\lambda_Y=2$, i.e. $Y=Y_H\times Y_P$. Since every basic set of $\mathcal{A}_{U\times P}$ outside $(L\times P)\cup U$ is rationally conjugate to $Y$ (Lemma~\ref{conj}) and $\mathcal{A}_{L\times P}=\mathbb{Z}L\otimes \mathcal{A}_P$ (Lemma~\aref{A-tenspr}(2)), we obtain
\begin{equation}\label{tensy}
\mathcal{A}_{V}=\mathcal{A}_{U}\otimes \mathcal{A}_P, 
\end{equation}
where $V=U\times P$. 

Let $\widehat{\mathcal{A}}$ be the $S$-ring over $\widehat{G}$ dual to $\mathcal{A}$. Then the Sylow $2$- and $p$-subgroups $P^\bot$ and $H^\bot$ are $\widehat{\mathcal{A}}$-subgroups by Lemma~\aref{A-dual}(1). From the definition of $K_H$ it follows that $\mathcal{A}_H=\cyc(K_H,H)\cong \mathbb{Z}C_2\wr \mathbb{Z}C_2\wr \mathbb{Z}C_2$. So 
\begin{equation}\label{dual8p1}
\widehat{\mathcal{A}}_{P^\bot}=\widehat{\mathcal{A}_{G/P}}\cong \mathbb{Z}C_2\wr \mathbb{Z}C_2\wr \mathbb{Z}C_2\cong \mathcal{A}_H
\end{equation}
by Statements~$(2)$ and~$(4)$ of Lemma~\aref{A-dual}.

Due to Lemma~\aref{A-dual}(2), we have $\widehat{\mathcal{A}}_{\widehat{G}/V^\bot}=\widehat{\mathcal{A}_V}$. Lemma~\aref{A-dual}(3) and Eq.~\eqref{tensy} yield that $\widehat{\mathcal{A}}_{\widehat{G}/V^\bot}$ is a tensor product of $S$-rings over groups of orders~$4$ and~$p$. Together with Lemma~\aref{A-intersection}, this implies that
\begin{equation}\label{dual8p2}
\lambda_{\widehat{X}}=|\widehat{X}\cap P^\bot\widehat{x}|\geq 2
\end{equation}
for every $\widehat{X}\in \mathcal{S}(\widehat{\mathcal{A}})$ outside $L^\bot\geq V^\bot$ of order~$4p$ and every $\widehat{x}\in \widehat{X}$. From Eqs.~\eqref{dual8p1} and~\eqref{dual8p2}, it follows that $\widehat{\mathcal{A}}$ is Cayley isomorphic to an $S$-ring from Case~$1$ or~$2$ and hence $\widehat{\mathcal{A}}$ is cyclotomic. Thus, $\mathcal{A}$ is cyclotomic by Lemma~\aref{A-dual}(5).
\end{proof}

Theorem~\ref{8p} immediately follows from Propositions~\ref{hnot8p},~\ref{pnot8p}, and~\ref{hpasubgroups8p}.

\section{Proof of Theorem~\ref{main1}}

\begin{lemm}\label{sections3}
Every proper subgroup of $G$ is a Schur group. 
\end{lemm}

\begin{proof}
Let $N\lneq G$. If $N$ is cyclic, then $N$ is a Schur group by~\cite[Theorem~1.1]{EKP1}, whereas if $N\cong E_4$ or $N\cong E_8$, then $N$ is a Schur group by~\cite[Theorem~1.2]{EKP2}. Otherwise, $N\cong E_4 \times C_p$ is a Schur group by~\cite[Theorem~1.5]{EKP2}.
\end{proof}

Let $\mathcal{A}$ be a nontrivial $S$-ring over $G$. Let us prove that $\mathcal{A}$ is schurian. Theorem~\ref{8p} holds for $\mathcal{A}$. We are done if $\mathcal{A}$ is cyclotomic. If $\mathcal{A}$ is a nontrivial tensor product, then $\mathcal{A}$ is schurian by Lemma~\aref{A-schurtens} and Lemma~\ref{sections3}. So we may assume that $\mathcal{A}$ is a nontrivial $S$-wreath product for some $\mathcal{A}$-section $S=U/L$ and one of Statements~$(1)$-$(4)$ from Theorem~\ref{8p} holds. The $S$-rings $\mathcal{A}_U$ and $\mathcal{A}_{G/L}$ are schurian by Lemma~\ref{sections3}. If Statement~$(1)$ or~$(2)$ holds, i.e. $|S|\leq 2$ or $|S|=4$ and $\mathcal{A}_S\neq \mathcal{T}_S$, then $\mathcal{A}_S$ is $2$-minimal by Lemma~\aref{A-2minsmall} and hence $\mathcal{A}$ is schurian by Lemma~\aref{A-2min}. If Statement~$(3)$ holds, i.e. $\mathcal{A}_{L}$ is $\otimes$-complemented in $\mathcal{A}_{U}$ or $\mathcal{A}_{S}$ is $\otimes$-complemented in $\mathcal{A}_{G/L}$, then $\mathcal{A}$ is schurian by Lemma~\aref{A-otimescomplement}. 

Now suppose that Statement~$(4)$ holds, i.e. $H$ and $P$ are $\mathcal{A}$-subgroups and $|U|=|G/L|=4p$. In this case, $S\cong C_{2p}$. As $H$ and $P$ are $\mathcal{A}$-subgroups, $H\cap U$ and $P$ are $\mathcal{A}_U$-subgroups of orders~$4$ and~$p$, respectively, and $H/L$ and $P/L$ are $\mathcal{A}_{G/L}$-subgroups of orders~$4$ and~$p$, respectively. Therefore $\mathcal{A}_U$ and $\mathcal{A}_{G/L}$ are cyclotomic by~\cite[Case~3,~p.~115]{EKP2} and hence $\mathcal{A}_S$ is also cyclotomic. Moreover, $\mathcal{A}_S$ is Cayley minimal by Lemma~\aref{A-cyclcayleymin}. Thus, $\mathcal{A}$ is schurian by Lemma~\ref{cayleymin}.

\section{$S$-rings over $C_3\times C_{3^k}$}

Let $k\geq 2$, $A\cong C_{3^k}$, $B\cong C_3$, and $G=B \times A \cong C_3\times C_{3^k}$. By the \emph{radical} $\rad(\mathcal{A})$ of an $S$-ring $\mathcal{A}$ over $G$, we mean the group generated by the groups $\rad(X)$, where $X$ runs over all basic sets of $\mathcal{A}$ containing an element of order~$3^k$. The lemma below providing a characterization of $S$-rings over $G$ collects information from~\cite[Theorem~4.1,~Theorem~5.1,~Theorem~6.1,~Lemma~6.1]{Ry2}.

\begin{lemm}\label{c3c3k}
Let $\mathcal{A}$ be a nontrivial $S$-ring over $G$. Then one of the following statements holds:
\begin{enumerate}

\tm{1} $|\rad(\mathcal{A})|=1$ and $\mathcal{A}=\mathcal{A}_L\otimes \mathcal{T}_U$, where $L$ and $U$ are $\mathcal{A}$-subgroups of orders~$3$ and~$3^k$, respectively;

\tm{2} $|\rad(\mathcal{A})|=1$ and $\mathcal{A}$ is cyclotomic;

\tm{3}  $|\rad(\mathcal{A})|>1$ and $\mathcal{A}$ is a nontrivial $U/L$-wreath product for some $\mathcal{A}$-section $U/L$ such that $|\rad(\mathcal{A}_U)|=1$.

\end{enumerate}
\end{lemm}

The subgroup of $A$ of order~$3$ is denoted by~$A_1$. Clearly, $A_1$ is a unique characteristic subgroup of $G$ of order~$3$. The image of $H\leq G$ under the canonical epimorphism from $G$ to $G/A_1$ is denoted by $\overbar{H}$.

\begin{lemm}\label{quotc3c3k}
Let $\mathcal{A}$ be an $S$-ring over $G$ satisfying Statement~$(2)$ of Lemma~\ref{c3c3k}, and $L$ an $\mathcal{A}$-subgroup of order~$3$. Then one of the following statements holds:
\begin{enumerate}

\tm{1} $\mathcal{A}_{G/L}$ is $2$-minimal;

\tm{2} $L=A_1$ and $\rad(\mathcal{A}_{G/L})=\overbar{B}$. 
\end{enumerate}

\end{lemm}

\begin{proof}
From~\cite[Theorem~5.1]{Ry2} it follows that $\mathcal{A}=\cyc(K,G)$, where $K\leq \aut(G)$ is one of the groups $K_0$-$K_9$ from~\cite[Table~1]{Ry2}. If $K\in\{K_0,\dots,K_5\}$, then $\mathcal{A}_{G/L}$ is $2$-minimal by~\cite[Corollary~5.2]{Ry2}. If $K\in \{K_6,K_7,K_8,K_9\}$, then from the definition of $K$ (see~\cite[Table~1]{Ry2}) it follows that $A_1$ is a unique $\mathcal{A}$-subgroup of order~$3$, i.e. $L=A_1$, and $\rad(\mathcal{A}_{G/A_1})=\overbar{B}$.  
\end{proof}

\begin{lemm}\label{radicalsc3c3k}
Let $\mathcal{A}$ be an $S$-ring over $G$ and $X\in \mathcal{S}(\mathcal{A})$ such that $\langle X \rangle\lneq G$ and $\rad(X)$ has a noncharacteristic subgroup $L$ of order~$3$. Then $\rad(Y)\geq L$ for every $Y\in \mathcal{S}(\mathcal{A})_{G\setminus \langle X \rangle}$. 
\end{lemm}

\begin{proof}
Let $V=\langle X \rangle$. As $\rad(X)\geq L$, we have $|V|\geq 9$. Since $L\leq \rad(X)$, the group $V$ is isomorphic to $C_3\times C_{3^l}$ for some $l\in\{1,\ldots,k-1\}$. Assume the contrary that there is $Y\in \mathcal{S}(\mathcal{A})_{G\setminus V}$ such that $L\nleq\rad(Y)$. Put $N=\rad(Y)$ and $U=\langle Y \rangle$. As $Y$ lies outside $V$, we conclude that $Y$ consists of elements of order greater than $3^l$. Together with $L\leq \rad(X)$, this yields that $X\cap U\neq \varnothing$ and hence $X\subseteq U$. The latter implies that $U\cong C_3\times C_{3^m}$ for some $m\in\{l+1,\ldots,k\}$. So every basic set of $\mathcal{A}_U$ containing an element of order $3^m$ is rationally conjugate to $Y$ by Lemma~\aref{A-burn} and consequently has the same radical~$N$. Therefore $\rad(\mathcal{A}_U)=N$.

The image of $T\subseteq G$ under the canonical epimorphism from $G$ to $G/N$ is denoted by $\widetilde{T}$. By the assumption, $N\ngeq L$. Therefore $U/N$ is noncyclic and 
\begin{equation}\label{radquot}
\rad(\widetilde{X})\geq \widetilde{L}. 
\end{equation}
In view of $\widetilde{V}<\widetilde{U}$ and Eq.~\eqref{radquot}, we have that $|\widetilde{U}|\geq 27$ and hence Lemma~\ref{c3c3k} holds for $\mathcal{A}_{\widetilde{U}}$. Note that $|\rad(\mathcal{A}_{\widetilde{U}})|=1$ because $\rad(\mathcal{A}_U)=N$. So one of Statements~$(1)$-$(2)$ of Lemma~\ref{c3c3k} holds for $\mathcal{A}_{\widetilde{U}}$. If Statement~$(1)$ holds, then a radical of every basic set of $\mathcal{A}_{\widetilde{U}}$ is trivial, a contradiction to Eq.~\eqref{radquot}. If Statement~$(2)$ holds, then  a description of all cyclotomic $S$-rings with trivial radical over $\widetilde{U}$ from~\cite[Theorem~5.1]{Ry2} implies that a radical of each basic set of $\mathcal{A}_{\widetilde{U}}$ is contained in a characteristic subgroup of order~$3$, a contradiction to Eq.~\eqref{radquot}. 
\end{proof}

\section{$S$-rings over $C_{6}\times C_{3^k}$ and $E_9 \times C_{2q}$}

In this section, we characterize $S$-rings over the groups $C_{6}\times C_{3^k}$ and $E_9 \times C_{2q}$. Let $G\cong C_6 \times C_{3^k}$ for some integer $k\geq 1$ or $G\cong E_9\times C_{2q}$ for some odd prime $q$. Clearly, $G=H\times P$, where $H$ is a Hall $2^\prime$-subgroup of $G$ and $P\cong C_2$.

\begin{theo}\label{2odd}
Let $\mathcal{A}$ be a nontrivial $S$-ring over $G$, $H_1$ a maximal $\mathcal{A}$-subgroup contained in $H$, and $P_1$ the least $\mathcal{A}$-subgroup containing $P$. Then one of the following statements holds:
\begin{enumerate}
\tm{1} $H_1P_1=G$, $|H_1|>1$, and $\mathcal{A}=\mathcal{A}_{H_1}\star\mathcal{A}_{P_1}$;

\tm{2} $|H_1|>1$ and $\mathcal{A}=\mathcal{A}_{H_1}\wr \mathcal{T}_{G/H_1}$;

\tm{3} $\mathcal{A}$ is the nontrivial $(H_1P_1)/P_1$-wreath product and $\mathcal{A}_{P_1}$ is $\otimes$-complemented in $\mathcal{A}_{H_1P_1}$;

\tm{4} $\mathcal{A}$ is the nontrivial $(H_1P_1)/P_1$- and $H_1/(H_1\cap P_1)$-wreath products.

\end{enumerate} 
\end{theo}

\begin{proof}
Due to Lemma~\aref{A-nonpowernew2}, we have $\mathcal{A}_{H_1P_1}=\mathcal{A}_{H_1}\star\mathcal{A}_{P_1}$. If $H_1=H$, then $H_1P_1=G$ and Statement~$(1)$ of the theorem holds as required. So we may assume that $H_1\lneq H$. Then Lemma~\aref{A-nonpowernew1} holds for $\mathcal{A}$. If Statement~$(1)$ of Lemma~\aref{A-nonpowernew1} holds for $\mathcal{A}$, i.e. $H_1P_1=G$, $P_1\lneq G$, and $\mathcal{A}=\mathcal{A}_{H_1} \star \mathcal{A}_{P_1}$, then Statement~$(1)$ of the theorem holds as desired. Suppose that Statement~$(2)$ of Lemma~\aref{A-nonpowernew1} holds for $\mathcal{A}$, i.e. $\mathcal{A}=\mathcal{A}_{H_1}\wr \mathcal{A}_{G/H_1}$ with $\mathcal{A}_{G/H_1}=\mathcal{T}_{G/H_1}$. Together with the condition that $\mathcal{A}$ is nontrivial, this yields that Statement~$(2)$ of the theorem holds as desired.  

Suppose that Statement~$(3)$ of Lemma~\aref{A-nonpowernew1} holds for $\mathcal{A}$, i.e. $\mathcal{A}$ is the nontrivial $(H_1P_1)/P_1$-wreath product. If $H_1\cap P_1$ is trivial, then $\mathcal{A}_{H_1P_1}=\mathcal{A}_{H_1}\star\mathcal{A}_{P_1}=\mathcal{A}_{H_1}\otimes\mathcal{A}_{P_1}$. This means that $\mathcal{A}_{P_1}$ is $\otimes$-complemented in $\mathcal{A}_{H_1P_1}$ and hence Statement~$(3)$ of the theorem holds. If $H_1\cap P_1$ is nontrivial, then $\mathcal{A}_{H_1P_1}=\mathcal{A}_{H_1}\star\mathcal{A}_{P_1}$ is the nontrivial $H_1/(H_1\cap P_1)$-wreath product and consequently so is $\mathcal{A}$. Thus, Statement~$(4)$ of the theorem holds and we are done.
\end{proof}

Since the star product from Statement~$(1)$ of Theorem~\ref{2odd} is a nontrivial tensor or generalized wreath product (see, e.g.,~\cite[Section~3.3]{Ry5}), Theorem~\ref{main15} for all required groups except for $E_9\times C_4$ immediately follows from Theorem~\ref{8p} and Theorem~\ref{2odd}. Theorem~\ref{main15} for $E_9\times C_4$ can be verified by computer calculation using the package~\cite{GAP}.

\section{Proof of Theorem~\ref{main2}}

The group $E_9\times C_4$ is a Schur group by the computational results~\cite{Ziv}. Further, we assume that $G\cong C_6 \times C_{3^k}$ for some integer $k\geq 1$ or $G\cong E_9\times C_{2q}$ for some odd prime $q$.

\begin{lemm}\label{sections}
Every proper subgroup of $G$ is a Schur group. 
\end{lemm}

\begin{proof}
Let $N\lneq G$. If $N$ is cyclic, then $N$ is a Schur group by~\cite[Theorem~1.1]{EKP1}. If $G\cong E_9\times C_{2q}$ and $N$ is noncyclic, then $N$ is isomorphic to one of the groups $E_9$, $E_9\times C_2$, $E_9\times C_q$. The first of them is a Schur group by~\cite[Theorem~1.2]{EKP2}, whereas the second and the third ones by~\cite[Theorem~1.1]{PR}. If $G\cong C_6\times C_{3^k}$ and $N$ is noncyclic, then $N$ is isomorphic to one of the groups $C_3\times C_{3^l}$, $l\leq k$, $C_6\times C_{3^m}$, $m<k$. The first of them is a Schur group by~\cite[Theorem~1.1]{Ry2}, whereas the second one by induction on~$m$ whose base follows from the computational results~\cite{Ziv}.
\end{proof}

Further throughout the section, we use the notation from Section~8. Let $\mathcal{A}$ be a nontrivial $S$-ring over $G$. Let us prove that $\mathcal{A}$ is schurian. Due to Lemma~\ref{sections}, we assume further that $\mathcal{A}_S$ is schurian for every $\mathcal{A}$-section $S\neq G$.

\begin{lemm}\label{schur2odd}
With the above notation, $\mathcal{A}$ is schurian unless $\mathcal{A}$ is the nontrivial $(H_1P_1)/P_1$- and $H_1/(H_1\cap P_1)$-wreath products. 
\end{lemm}

\begin{proof}
One of the statements of Theorem~\ref{2odd} holds for $\mathcal{A}$. If Statement~$(1)$, or~$(2)$, or~$(3)$ holds, then $\mathcal{A}$ is schurian by Lemma~\ref{schurstar}, or Lemma~\aref{A-2min}, or Lemma~\aref{A-otimescomplement}, respectively. If Statement~$(4)$ holds, then $\mathcal{A}$ is the nontrivial $(H_1P_1)/P_1$- and $H_1/(H_1\cap P_1)$-wreath products and we are done.
\end{proof}

In view of Lemma~\ref{schur2odd}, we may assume that $\mathcal{A}$ is the nontrivial $S_1=U_1/P_1$- and $S_0=H_1/L_0$-wreath products, where $U_1=H_1P_1$ and $L_0=H_1\cap P_1$. In particular, $H_1<H$.

If $G\cong E_9\times C_{2q}$, then 
$$|S_0|\in\{1,3,q\}$$
because $|H_1|<|H|=9q$. If $|S_0|\leq 3$, then $\mathcal{A}_{S_0}$ is $2$-minimal by Lemma~\aref{A-2minsmall} and $\mathcal{A}$ is schurian by Lemma~\aref{A-2min}. If $|S_0|=q$, then $H_1\cong C_{3q}$, $L_0\cong C_3$, and $G/L_0\cong C_{6q}$. Therefore $\mathcal{A}$ is schurian by Lemma~\aref{A-prime}.

Now let $G\cong C_6 \times C_{3^k}$. In this case, $H\cong C_3 \times C_{3^k}$.

\begin{lemm}\label{cycl3section}
If $\mathcal{A}$ is a nontrivial $S=U/L$-wreath product for some $\mathcal{A}$-subgroups $L\leq U$ such that $U$ or $G/L$ is a cyclic $3$-group, then $\mathcal{A}$ is schurian. 
\end{lemm}

\begin{proof}
To prove the lemma, it suffices to show that $\mathcal{A}$ is a nontrivial $T=V/N$-wreath product for some $\mathcal{A}$-subgroups $N\leq V$ such that $\mathcal{A}_V$ or $\mathcal{A}_{G/N}$ is an $S$-ring with trivial radical over a cyclic $3$-group. Indeed, in this case one of the above $S$-rings is trivial or normal by Lemma~\aref{A-leungman}. In the former case, $|T|=1$ and hence $\mathcal{A}$ is schurian by Lemma~\aref{A-2min}. In the latter one, if $|T|\leq 3$, then $\mathcal{A}_T$ is $2$-minimal by Lemma~\aref{A-2minsmall}, whereas if $|T|\geq 9$, then $\mathcal{A}_T$ is normal by~\cite[Corollary~4.4]{EP2} and hence $\mathcal{A}_T$ is $2$-minimal by Lemma~\aref{A-2minnorm}. Thus, $\mathcal{A}$ is schurian by Lemma~\aref{A-2min}.

In view of the above paragraph and Lemma~\aref{A-leungman}, we may assume that $\mathcal{A}_U$ or $\mathcal{A}_{G/L}$ is a nontrivial generalized wreath product over a cyclic $3$-group. In the former case, $\mathcal{A}_U$ is the nontrivial $V_1/M_1$-wreath product, where $M_1$ is the least $\mathcal{A}_{U}$-subgroup and $|\rad(\mathcal{A}_{V_1})|=1$, by Lemma~\ref{cyclpwreath}(1). Clearly, $M_1\leq L$ and hence $\mathcal{A}$ is the $V_1/M_1$-wreath product. Therefore we are done by the previous paragraph. 

In the latter case, Lemma~\ref{cyclpwreath}(2) implies that $\mathcal{A}_{G/L}$ is the nontrivial $V_0/M_0$-wreath product, where $V_0$ is the greatest proper $\mathcal{A}_{G/L}$-subgroup and $|\rad(\mathcal{A}_{(G/L)/M_0)}|=1$. It is easy to see that $V_0\geq S$ and hence $\mathcal{A}$ is the $V_0^{\pi^{-1}}/M_0^{\pi^{-1}}$-wreath product, where $\pi$ is the canonical epimorphism from $G$ to $G/L$. The $S$-ring $\mathcal{A}_{G/M_0^{\pi^{-1}}}$ is isomorphic to $\mathcal{A}_{(G/L)/M_0}$. Therefore $\rad(\mathcal{A}_{G/M_0^{\pi^{-1}}})$ is trivial and again we are done by the previous paragraph. 
\end{proof}

\begin{lemm}\label{noncharrad}
If $G/P_1$ is noncyclic and there is a basic set of $\mathcal{A}_{U_1/P_1}$ whose radical has a noncharacteristic subgroup of $G/P_1$ of order~$3$, then $\mathcal{A}$ is schurian. 
\end{lemm}

\begin{proof}
Let $X\in \mathcal{S}(\mathcal{A}_{U_1/P_1})$ be such that $\rad(X)\geq L$, where $L$ is a noncharacteristic subgroup of $G/P_1$ of order~$3$. Lemma~\ref{radicalsc3c3k} implies that $L\leq \rad(Y)$ for every $Y\in \mathcal{S}(\mathcal{A}_{G/P_1})$ outside $\langle X \rangle$ and hence outside $U_1/P_1$. Since $P_1\leq \rad(Y)$ for every $Y\in \mathcal{S}(\mathcal{A})_{G\setminus U_1}$, we conclude that $L^{\pi^{-1}}\leq \rad(Y)$ for every $Y\in \mathcal{S}(\mathcal{A})_{G\setminus U_1}$, where $\pi$ is the canonical epimorphism from $G$ to $G/P_1$. Therefore $\mathcal{A}$ is the $U_1/L^{\pi^{-1}}$-wreath product. Note that $G/L^{\pi^{-1}}$ is a cyclic $3$-group because $L$ is noncharacteristic subgroup of $G/P_1$ of order~$3$. Thus, $\mathcal{A}$ is schurian by Lemma~\ref{cycl3section}. 
\end{proof}

\begin{lemm}\label{trivradU}
If $\mathcal{A}$ is a nontrivial $S=U/L$-wreath product for some $\mathcal{A}$-subgroups $L\leq U$ such that $L\leq P_1$, $U\leq H_1$, $U\cong C_3\times C_{3^l}$, $l\geq 2$, and $|\rad(\mathcal{A}_U)|=1$, then $\mathcal{A}$ is schurian.  
\end{lemm}

\begin{proof}
Let $A$ and $B$ be the subgroups of $U$ of orders~$3^l$ and~$3$, respectively, such that $U=B\times A$ and $A_1$ the subgroup of $A$ of order~$3$. If $U=L$, in particular, if $\mathcal{A}_U$ is trivial, then $|S|=1$ and we are done by Lemma~\aref{A-2min}. So further, we assume that $\mathcal{A}_U$ is nontrivial and hence one of Statements~$(1)$-$(2)$ of Lemma~\ref{c3c3k} holds for $\mathcal{A}_U$. If Statement~$(1)$ holds, then $\mathcal{A}_{U}=\mathcal{A}_{L}\otimes \mathcal{A}_V$ for some $\mathcal{A}_{U}$-subgroup $V$. Therefore $\mathcal{A}_{L}$ is $\otimes$-complemented in $\mathcal{A}_{U}$ and hence $\mathcal{A}$ is schurian by Lemma~\aref{A-otimescomplement}.

Suppose that Statement~$(2)$ of Lemma~\ref{c3c3k} holds for $\mathcal{A}_{U}$, i.e $\mathcal{A}_{U}$ is cyclotomic with trivial radical. Since $\mathcal{A}_{U}$ is cyclotomic, the characteristic subgroups $A_1$ and $E$ of $U$, where $E$ consists of all elements of order~$3$ from $U$, are $\mathcal{A}_{U}$-subgroups. So $L\cap E$ is a nontrivial $\mathcal{A}_{U}$-subgroup. Moreover, $L\cap E$ has an $\mathcal{A}_{U}$-subgroup $N$ of order~$3$. Indeed, if $A_1\leq L$, then one can put $N=A_1$, whereas otherwise $|L|=3$ and one can put $N=L$.

Note that $\mathcal{A}$ is the nontrivial $U/N$-wreath product because $N\leq L$. We are done by Lemma~\aref{A-2min} if $\mathcal{A}_{U/N}$ is $2$-minimal. So in view of Lemma~\ref{quotc3c3k}, we may assume that $N=A_1$ and 
\begin{equation}\label{radB}
\rad(\mathcal{A}_{U/A_1})=\overbar{B},
\end{equation}
where $\overbar{B}$ is the image of $B$ under the canonical epimorphism from $G$ to $G/A_1$. If $G/P_1$ is cyclic, then $\mathcal{A}$ is schurian by Lemma~\ref{cycl3section}. Otherwise, the inclusions $A_1\leq L\leq P_1$, $U\leq H_1\leq U_1$ and Eq.~\eqref{radB} imply that $\mathcal{A}$ satisfies the conditions of Lemma~\ref{noncharrad} and hence $\mathcal{A}$ is schurian. 
\end{proof}

Recall that $\mathcal{A}$ is the nontrivial $S_0=H_1/L_0$-wreath product. If $|H_1|\leq 9$, then $|S_0|\leq 3$ and hence $\mathcal{A}_{S_0}$ is $2$-minimal by Lemma~\aref{A-2minsmall}. So $\mathcal{A}$ is schurian by Lemma~\aref{A-2min}. Let $|H_1|\geq 27$. In view of Lemma~\ref{cycl3section} and Lemma~\ref{trivradU}, we may assume that $H_1$ is noncyclic and $|\rad(\mathcal{A}_{H_1})|>1$. Then $\mathcal{A}_{H_1}$ is a nontrivial $U/L$-wreath product for some $\mathcal{A}_{H_1}$-subgroups $L\leq U$ such that $|\rad(\mathcal{A}_U)|=1$ by Lemma~\ref{c3c3k}.

Suppose that $L_0\cap L$ is nontrivial. Then $\mathcal{A}$ is the nontrivial $S=U/(L_0\cap L)$-wreath product. If $|U|\leq 9$, then $|S|\leq 3$. So $\mathcal{A}_{S}$ is $2$-minimal by Lemma~\aref{A-2minsmall} and hence $\mathcal{A}$ is schurian by Lemma~\aref{A-2min}. Let $|U|\geq 27$. Then  $\mathcal{A}$ is schurian by Lemma~\ref{cycl3section} if $U$ is cyclic and by Lemma~\ref{trivradU} otherwise.

Now suppose that $L_0\cap L$ is trivial. Then $L_0$ or $L$ is a noncharacteristic subgroup of $H_1$ of order~$3$. Recall that $\mathcal{A}$ is the $U_1/P_1$-wreath product. If $G/P_1$ is cyclic, then $\mathcal{A}$ is schurian by Lemma~\ref{cycl3section}. So we may assume that $G/P_1$ is noncyclic and hence $L_0$ which is a subgroup of $P_1$ can not be a noncharacteristic subgroup of order~$3$. Therefore $L$ is a noncharacteristic subgroup of order~$3$. Recall that $L\leq \rad(X)$ for every $X\in \mathcal{S}(\mathcal{A})_{H_1\setminus U}$ because $\mathcal{A}_{H_1}$ is the nontrivial $U/L$-wreath product. Since $G/P_1$ is noncyclic, $\overbar{L}\leq \rad(\overbar{X})$ for every $X\in \mathcal{S}(\mathcal{A})_{H_1\setminus U}$, where $\overbar{L}$ and $\overbar{X}$ are the images of $L$ and $X$, respectively, under the canonical epimorphism from $G$ to $G/P_1$. Thus, $\mathcal{A}$ satisfies the conditions of Lemma~\ref{noncharrad} and consequently $\mathcal{A}$ is schurian.

\end{document}